\documentclass[11pt, reqno]{amsart}
\usepackage{a4wide}
\usepackage{hyperref}
\usepackage{amsmath}
\usepackage{paralist}
\usepackage{amssymb}
\usepackage{amsthm}
\usepackage{amscd}
\usepackage{graphicx,mathrsfs}
\usepackage{fontenc}
\usepackage{enumitem}

\begin{document}
\title[removable singularities]
{Removable singularities and Harnack inequality for nonlinear H\"ormander degenerate subelliptic equations}

\author[ Jiayi Qiang, Yawei Wei, Mengnan Zhang]
{ Jiayi Qiang, Yawei Wei, Mengnan Zhang}

\address{Jiayi Qiang \newline
School of Mathematical Sciences\\ Nankai University\\ Tianjin 300071 China}
\email{2120230051@mail.nankai.edu.cn}

\address{Yawei Wei \newline
School of Mathematical Sciences and LPMC\\ Nankai University\\ Tianjin 300071 China}
\email{weiyawei@nankai.edu.cn}

\address{ Mengnan Zhang  \newline
School of Mathematical Sciences\\ Nankai University\\ Tianjin 300071 China}
\email{1120220030@mail.nankai.edu.cn}
\thanks{Acknowledgements: This work is supported by the NSFC under the grands 12271269, and  the Fundamental Research Funds for the Central Universities.}



\begin{abstract}
 This paper concerns the quasilinear subelliptic function derived from H\"ormander vector fields.  Based on the significant  work of J. Serrin in \cite{SER}, M. Meier in \cite{MM1}, and  L. Capogna, D. Danielli and N. Garofalo in \cite{LC1,LDN}, we obtain the removable singularities and Harnack inequality by a sharp Sobolev inequalities under weaker integrability of coefficients in structure conditions. Furthermore, we get the H\"older continuity when domain $\Omega$ is equiregular.
\end{abstract}

\maketitle
\numberwithin{equation}{section}
\newtheorem{Thm}{Theorem}[section]
\newtheorem{Rem}{Remark}[section]
\newtheorem{lemma}{Lemma}[section]
\newtheorem{corollary}{Corollary}[section]
\newtheorem{assumption}{Assumption}[section]
\newtheorem{definition}{Definition}[section]
\newcommand{\R}{\mathbb{R}}
\newcommand{\B}{\mathbf{B}}
\newcommand{\e}{{\mathcal {E}}}
\newcommand{\M}{{\mathcal{H}}}
\newcommand{\Y}{{\mathbf{P}}}
\allowdisplaybreaks
\section{Introduction}

In this paper, let $U$ be  an connected open domain of $\mathbb{R}^{n}$ with $n\geq 2$.  Consider a system of real smooth vector
fields $ X:=(X_1,X_2,...,X_m)$ defined on $U$, satisfying   the H\"ormander's condition, see Definition \ref{H}.  We study a kind of nonilinear degenerate subelliptic equations as follows,

\begin{equation}\label{equa}
\sum_{i=1}^{m} X^{*}_{i}A_{i}(x,u,{X}u)= B(x,u,{X}u), \qquad x \in   \Omega.
\end{equation}
Here $\Omega $ is a bounded open subset of $U$ and $X^*_{i}$ is the formal adjoint of $X_i$.  Since $X_i$ is  a real smooth vector field  defined in $U$.
 then $X_i$ and $X_i^{*}$ can be seen as a differential operator, $i=1,2,\cdots, m$,
\begin{equation}\label{eq:26}
X_i=\sum\limits_{j=1}^{n}h_{ij}(x)\partial_{x_j},
\end{equation}
\begin{equation}\label{eq:26}X^{*}_i \varphi =-\sum\limits_{j=1}^{n}\partial_{x_j}(h_{ij}(x)\varphi).\end{equation}
with $h_{ij}\in C^{\infty}(U)$.

The   functions $A$ and $B$ are defined on $ \Omega\times\mathbb{R}\times\mathbb{R}^n$, with values in $\mathbb{R}^n$ and $\mathbb{R}$, respectively, and $A=(A_1,\cdots,A_m)$.
We assume that
equation \eqref{equa}  satisfies the following structure conditions:
\begin{equation}\label{cond}
\left\{
\begin{aligned}
&|A(x,u,{X}u)|\leq a |{X}u|^{p-1}+b(x)|u|^{p-1}+e(x), \\
&{X}u \cdot A(x,u,{X}u)\geq |{X}u|^p-d(x)|u|^p-g(x),
\end{aligned}
\right.
\end{equation}
\begin{equation}\tag{\theequation$'$}\label{cond_prime}
 |B(x,u,Xu)|  \leq  c_0|Xu|^p+c(x)|Xu|^{p-1}+d(x)|u|^{p-1}+f(x).
\end{equation}
where  $1<p\leq \tilde{v}$, and $\tilde{v}$ is the generalized M\'etivier index of $\Omega$, see \eqref{eq:31}. In the foregoing estimates,  $a>0, c_0 \geq 0$ are constants, and the functions $ b,c,d,e,f,g$,  defined on $\Omega$,  are measurable, almost everywhere nonnegative, and   belong to the following Lebesgue classes:
\begin{equation}\label{func}
b,e \in \left\{
\begin{aligned}
&L^{\frac{\tilde{v}}{p-1}} (\Omega)  \quad  &if \;p<\tilde{v}, \\
&L^{\frac{\tilde{v}}{\tilde{v}-1-\varepsilon}} (\Omega)  \quad  &if \;p=\tilde{v},
\end{aligned}
\right.
\quad c \in L^{\frac{\tilde{v}}{1-\varepsilon}} (\Omega), \quad d,f,g \in L^{\frac{\tilde{v}}{p-\varepsilon}} (\Omega),
\end{equation}
for some $\varepsilon \in (0, 1) $.
Since the norm of $b,e$
 changes as the value of $p$ varies,  in the following article, the norm defined above is abbreviated as $|| b||_{\Omega}, ||e||_{\Omega}, || c||_{\Omega}, || d||_{\Omega}, || f||_{\Omega}, || g||_{\Omega}$.

 We give  some  notations and Definitions about H\"ormander vector fields.
   \begin{definition}[Commutator]
The commutator of two smooth vector fields
$$X=\sum_{j=1}^{n}a_j(x)\partial_{x_j};\ \  Y=\sum_{j=1}^{n}b_j(x)\partial_{x_j}$$
is defined as: $$[X,Y]=XY-YX=\sum_{i,j=1}^{n}(a_j\partial_{x_j}b_i-b_j\partial_{x_j}a_i)\partial_{x_i}.$$
 \end{definition}

    \begin{definition}[H\"ormander's condition] \label{H}
 If there exists a smallest positive integer $i_0$ such that the  smooth vector fields $X_1,X_2,...,X_m$ together with their commutators of length at most $i_0$ span $\mathbb{R}^{n}$ at each point in $U$, then we call $X=(X_1,X_2,...,X_m)$ satisfies the   H\"ormander's condition. We also refer to $X=(X_1,X_2,...,X_m)$ as H\"ormander's vector fields.
 \end{definition}
Let Lie($X$) be the Lie algebra generated by vector fields $X_1, X_2,\cdots, X_m$ over $\mathbb{R}$. For $l\in \mathbb{N}^{+}$, we define $$\textup{Lie}^{l}(X):=span\{[X_{i_1},\cdots, [X_{i_{j-1}},X_{i_j}]]\ |1\leq i_j\leq m,j\leq l\}.$$  The H\"ormander's condition gives that $\textup{Lie}(X)(x)=\{Z(x)|Z\in \textup{Lie}(X) \}=T_{x}(U)$ for all $x\in U$. This means, for each point $x\in U$, there  exists a minimal integer $i(x)\leq i_0$ such that
$$\textup{Lie}^{i(x)}(X)(x)=\{Z(x)|Z\in \textup{Lie}^{i(x)}(X) \}=T_{x}(U) .$$
The integer $i(x)$ is known as the degree of nonholonomy at $x$.

For $x\in U$ and $1\leq j\leq i(x)$, we set $V_{j}(x):=\textup{Lie}^{j}(X)(x)$. It follows that
$$ \{0\}=V_0(x)\subset V_1(x)\subset\cdots \subset V_{i(x)-1}(x) \subsetneq V_{i(x)}(x)=T_{x}(U).$$ Then, we define
\begin{equation}\label{eq:32} v(x):=\sum_{j=1}^{i(x)}j(v_j(x)-v_{j-1}(x))\end{equation} as the pointwise homogeneous dimension at $x$, and
 \begin{equation}\label{eq:31}\tilde{v}:=\max_{x\in \overline{\Omega}}v(x) \end{equation}
  as the generalized M\'etivier index of $\Omega$.  Next,  we introduce some notations to present the precise definition of local homogeneous dimension. Let $J=(j_1,\cdots,j_k)$ be  a multi-index with length $|J|=k$, where $1\leq j_i\leq m$ and $1\leq k\leq r$. We assign a commutator $X_J$ of length $k$ such that
  $$X_{J}:=[X_{j_1},[X_{j_2},\cdots[X_{j_{k-1},X_{j_k}}]\cdots]],$$
  and set $$X^{k}:=\{X_{J}|J=(j_1,\cdots,j_k),\ 1\leq j_i\leq m, \ |J|=k\} $$
the collection of all commutators of length $k$. Let $Y_{1},\cdots,Y_{l}$ be an enumeration of the components of $X^{(1)},\cdots,X^{(i_0)}$. We say $Y_i$ has formal degree $d(Y_i)=k$ if $Y_i$ is an element of $X^{k}$. If $I=(i_1,i_2,\cdots,i_n)$ is a n-tuple of integers with $1\leq i_k\leq l$, we define
$$ d(I):=d(Y_{i_1})+d(Y_{i_2})+\cdots+d(Y_{i_n}),$$
and the so-called Nagel-Stein-Wainger polynomial
$$ \Lambda(x,r):=\Sigma_{I}|\lambda_{I}(x)r^{d(I)}|,$$
where $\lambda_{I}(x):=\det(Y_{i_1},Y_{i_2},\cdots,Y_{i_n})(x)$, and $I=(i_1,i_2,\cdots,i_n)$ ranges in the set of n-tuples
satisfying $1\leq i_k\leq l$.

Let $\Omega\subset\subset U$ be a bounded open set.
 According to [\cite{LC1}, (3.4), p. 1166], the local homogeneous dimension $Q$ relative to the bounded set $\Omega$ is precisely defined as
follows:
 \begin{equation}\label{eq:41}
 Q:=\max\{d(I)|\lambda_{I}(x)\neq0 \ \text{and} \  x\in \overline{\Omega}\}=\sup\limits_{x\in \Omega}\left(\lim_{r\to +\infty}\frac{\ln \Lambda(x,r)}{\ln r}\right). \end{equation}

\begin{definition}[Equiregular]\label{equiregular}
 We call a point $x\in U$ is regular if, for every $1\leq j\leq i(x)$, the dimension $v_{j}(y)$ is
a constant as $y$ varies in an open neighbourhood of $x$. Moreover, for any subset $\Omega\subset U$, we say $\Omega$ is equiregular if every point of $\Omega$ is regular. For the equiregular connected subset $\Omega$, the pointwise homogeneous dimension $v(x)$ is a constant $v$.
\end{definition}

\begin{definition}[H\"ormander operators]\label{HO}
If $X_0, X_1,\cdots,X_m $ is a system of H\"ormander vector fields in $U$, then
\begin{equation}\label{eq:33}
L=\sum_{j=1}^{m} X_j^2+X_0\end{equation}
is called a H\"ormander operator in $U$. If $X_0\equiv0$, that is if $$L=\sum_{j=1}^{m} X_j^2,$$
 then $L$ is called a sum of squares of H\"ormander vector fields.\end{definition}

The vector field $X_0$ in \eqref{eq:33} is often called ``drift", from the physical interpretation of $L$ as a transport-diffusion operator, hence operators like  \eqref{eq:33} can also be called
``operators with drift", when we want to stress the effective presence of this term. The fact that the vector field $X_0$ is required, together with the $X_i$ $(i = 1, 2,\cdots, m)$, to fulfill H\"ormander's condition, means that the first order operator $X_0$ cannot be thought as a lower
order term, but must be considered as belonging to the principal part of $L$. In some sense, $X_0$
``weights" as a second order derivative, analogously to the time derivative in the heat operator $\partial^2_{xx}-\partial_ t$.

There are many H\"ormander operators, such as Grushin operator, Kohn Laplacian, Kolmogorov operator, and so on.
Next, we will consider  a type of Grushin type opterator in $\mathbb{R}^n$, which was also studied in \cite{BFI}.
Given
\begin{equation}\label{eq:35}X=\left(\partial {x_1},\cdots,\partial {x_h}, \left(\sum_{i=1}^{h} {x_i}^2\right)^{\frac{k'}{2}}\partial {z_1}, \cdots, \left(\sum_{i=1}^{h} {x_i}^2\right)^{\frac{k'}{2}}\partial {z_l}\right),\end{equation} with $h+l=n$,  $h\geq 2$, $l\geq 1$ and $k'\in \{\mathbb{N}\cup \{0\}\}$.
Then
 \begin{equation}\label{eq:34}
 P_{k'}:= -\sum_{i=1}^{h}\frac{\partial^2}{\partial {x_i}^2}- \left(\sum_{i=1}^{h} {x_i}^2\right)^{k'}\sum_{j=1}^{l}\frac{\partial^2}{\partial {z_j}^2},\end{equation}
is  a type of Grushin type opterator in $\mathbb{R}^n$.

The study of nonlinear degenerate elliptic equations derived from H\"ormander's vector fields has progressed significantly  since H\"ormander's celebrated work \cite{LH} on hypoellipticity, such as \cite{GBFEMS,GBF,LPR}, etc.  Among them,  Folland and Stein in \cite{GBFEMS} studied the regularity problem of the $\bar{\partial}_{b}$ operator on the Heisenberg group  and obtained the optimal subelliptic estimate.
Subsequently,  Folland in \cite{GBF}   generalized this precise result to the sub-Laplacian on stratified Lie groups.  Later, Rothschild  and Stein in \cite{LPR} established the  Rothschild-Stein  lifting and approximation theorem,  based on which  the fundamental study of degenerate elliptic operators has made great progress and development. For H\"ormander vector fields,  the Sobolev inequality has been studied in \cite{LC,LC1,GSP,LDN}, and the Poincar\'e-Wirtinger type inequality has been obtained in \cite{GL,DJ}. The existence of fundamental solutions and  related problems  have been researched widely in  \cite{BFI, SAM, ASC}, etc. In recent research work, we mention that Hua Chen, Hong-Ge Chen, and Jin-Ning Li in \cite{GSP} carried out  highly meaningful work, which helps advance the development of research on degenerate operators under general H\"ormander vector fields. Specifically, they  derived sharp Sobolev inequalities on $W_{X,0}^{k,p}(\Omega)$ defined in Definition \ref{Def2.2},  where the critical Sobolev exponent depends on the generalized M\'etivier index.  What's more, they
 established the isoperimetric inequality, logarithmic Sobolev inequalities, Rellich-Kondrachov compact embedding theorem, Gagliardo-Nirenberg
inequality, Nash inequality, and Moser-Trudinger inequality in the context of general H\"ormander vector fields.
 W. Bauer,   K. Furutani and   C. Iwasaki in  \cite{BFI} examined a class of Grushin type operators as \eqref{eq:34} using methods that rely on an appropriate coordinate transformation and incorporate the theory of Bessel functions, modified Bessel functions, and Weber's second exponential integral.   As a result, they explained  the  geometric framework, proved some analytic properties such as essential self-adjointness, and gave  an explicit expression of the fundamental solution.

The removable singularities  of weak solutions has still been extensively  studied.  For the classical elliptic equation,  Serrin in \cite{SER}  researched the removable singularities for weak solutions of nonlinear elliptic equation.  Later,  the work in  \cite{FN}  improved the removability results of Serrin and  obtained best possible conditions for removable singularity at the point for  solutions of nonlinear elliptic  equations of divergent form.   For  nonlinear degenerate  elliptic equations arising from H\"ormander vector fields,  the authors  in \cite{LC1}  established  sharp capacitary estimates for Carnot-Carath\'eodory rings,  and  the  removable singularities  of weak solutions was studied   on this basis.   Bo Wang in \cite{BW}  researched the  removable singularities for viscosity solutions to  degenerate elliptic Pucci operator on the Heisenberg group, where the second order term is
obtained by a composition of degenerate elliptic Pucci operator with the degenerate
Heisenberg Hessian matrix.  Shanming Ji, Zongguang Li and Changjiang Zhu in \cite{SJZL}  did an outstanding work with rich conclusions, which greatly extends the existing research findings. Focused  on the multi-dimensional Burgers equations on the whole space $\mathbb{R}^n$, they showed the removable singularity property of unbounded profiles. Apart from that, to understand the possible singularity and the asymptotic stability of unbounded profiles on the whole space $\mathbb{R}^n$,  they  introduced the generalized derivatives (distributions), and then  established  the asymptotic behavior of
large perturbations around the unbounded profiles.
 Finally we mention that the similar studies have been recently made for quasi(super)harmonic functions in \cite{AB} and  Navier-Stokes equations  in \cite{HKEU},  etc.

The Harnack inequality is an important research topic in the field of partial differential equations and has been widely studied.  The anothers in \cite{LDN}  gave  the Harnack inequality for non-negative  weak solutions of quasilinear degenerate equations derived from H\"ormander vector fields.  Similar results are also obtained in \cite{AM,GL1,LWPN,FF}, etc. Among them,  Ferrari in \cite{FF}  proved that Harnack inequality holds true for a very general class of two-weight subelliptic operators given by a system of H\"ormander vector fields. In  Resent years,  based on the study on Green functions of the subelliptic operators  in combination with the fundamental tools of Poincar\'e and Sobolev embeddings,  the Harnack inequality for weighted subelliptic $p$-Laplace equation with a potential term  was  obtained in \cite{TDD}.

For a kind of quasilinear degenerate elliptic equations as follows,
\begin{equation}\label{equa1}
-{div}\;A(x,u,\nabla u)= B(x,u,\nabla u), \qquad x \in   \Omega,
\end{equation}
where $\Omega $ is a bounded open subset of $\mathbb{R}^n$.
Under the conditions \eqref{cond}, \eqref{cond_prime} with $c_0=0$ and \eqref{func} with $\tilde{v}=n$ hold,
J. Serrin in \cite{SER} studied the   removable singularities and  Harnack inequality  of solutions of \eqref{equa1}  by the iteration technique introduced by Moser in references \cite{JM} and \cite{JM1},  while making  strong use of the general Sobolev inequalities.  M. Meier in \cite{MM1}  dealt with  removable singularities and  Harnack inequality  for weak solutions of quasilinear elliptic systems.  Even in the case of a single equation, M. Meier in \cite{MM1} generalized  previous work of J. Serrin, admitting that the nonlinearity $B(x,u,\nabla u)$ satisfies \eqref{cond_prime} with $c_0\geq0$ or \eqref{thmcond1}.

For the H\"ormander  vector fields $ X=(X_1,X_2,...,X_m)$ in $U$,  let \begin{equation}\label{eq:39}Lu=\sum^{m}_{i=1}X_i^{*}X_iu=0,\end{equation} where $X_i^{*}$ is the formal adjoint of $X_i$. L. Capogna, D. Danielli and N. Garofalo in \cite{LC1},  \cite{LDN} extended  the results in \cite{SER} to a kind of quasilinear degenerate subelliptic equations \eqref{equa} under the assumption $\text{(H)}$.

 $\text{(H)}$: there is a fundamental solution $\Gamma(x,y)$   of \eqref{eq:39} satisfying
 \begin{equation}\label{eq:42}C\frac{d(x,y)^2}{B(x,d(x,y))}\leq \Gamma(x,y)\leq C^{-1}\frac{d(x,y)^2}{B(x,d(x,y))},\end{equation}
\begin{equation}\label{eq:43}|X\Gamma(x,y)|\leq  C^{-1}\frac{d(x,y)}{B(x,d(x,y))}.\end{equation}
 They overcame the substantial difficulties in the problem  from the subelliptic geometry and  established the following embedding theorem of Sobolev type:
 Let  $\Omega\subset \mathbb{R}^n$ be a bounded open set and denote by $Q$ the local homogeneous dimension relative to $\Omega$ defined in \eqref{eq:41}. Given $1\leq p\leq Q$ there  exist $C >0$ and $R_0 > 0$, such that for any $x_0\in \Omega$, $r\leq R_0$, we have
 $$ \left(\frac{1}{B(x_0,r)}\int_{B(x_0,r)}|u|^{q}dx\right)^{\frac{1}{q}}\leq Cr   \left(\frac{1}{B(x_0,r)}\int_{B(x_0,r)}|Xu|^{p}dx\right)^{\frac{1}{p}}$$
 for $ 1\leq q\leq \frac{Qp}{Q-p}$, and all $u\in W_{X,0}^{1,p}\left((B(x_0,r)\right)$.
  Under the conditions \eqref{cond}, \eqref{cond_prime} (where $c_0=0$) and \eqref{func} (with $\tilde{v}$  replaced by  Q) hold, they also obtained the  removable singularities  and  Harnack inequality  for  solutions of  \eqref{equa}. Apart from that, the work in \cite{LC1} also established  sharp capacitary estimates for Carnot-Carath\'eodory rings, and studied  the  local behavior  of  solutions  to \eqref{equa} having a  singularity at one point.

This paper concerns the removable singularities and Harnack inequality  for  quasilinear subelliptic equation. Firstly, we study the removable singularities for weak solutions of \eqref{equa}. Secondly, we obtain the Harnack inequality for non-negative bounded weak solutions of \eqref{equa}. Furthermore, we get the  H\"older continuity  when  domain $\Omega$ is equiregular. Finally, we  give an application on higher step Grushin type operator.

Now, we present the main results of this article.
\begin{definition}\label{definition1}
Let $U_1 \subset U$ be a bounded open set, and $\Sigma \subset U_1 $ be a compact set. For $1\leq s <+\infty$ we define  the s-capacity  of the condenser $(\Sigma, U_1 )$ as
$$cap_{s}(\Sigma, U_1 )=\inf\{\int_{ U} |X \psi|^s dx: \psi \in \mathnormal{C}^{\infty}_{0}(U_1 ),\psi \geq 1 \; \text{on}\; \Sigma \}. $$
We call  $cap_{s}(\Sigma)=0$ if there is a bounded open set $U_1$ such that $cap_{s}(\Sigma, U_1 )=0$.

\end{definition}
\begin{Thm}\label{Thmremovsing}
Let $U$ be an connected open domain of $\mathbb{R}^{n}$, $\Omega\subset U$ be  a bounded open domain, and  $\Sigma \subset U$ be a compact set with $cap_s(\Sigma)=0$ for some $s\in [p, \tilde{v}]$,  where $\tilde{v}$ is the generalized M\'etivier index of $\Omega$. Assume that  \eqref{cond}, \eqref{func} are satisfied, and that
\begin{equation}\label{thmcond1}
u \cdot B(x,u,Xu)  \leq (1-\theta)|Xu|^p+|u|\{c(x)|Xu|^{p-1}+d(x)|u|^{p-1}+f(x)\}
\end{equation}
holds for $a.e. \; x \in \Omega-\Sigma$, where $ \theta \in (0,1]$.  For any fixed $\delta>0$, if $u\in \mathnormal{W}^{1,p}_{X,loc}(\Omega-\Sigma) \cap L_{\frac{s(p-\theta)}{s-p}(1+\delta),loc}(\Omega)$ is a weak solution of \eqref{equa} in $\Omega-\Sigma$ and   $B(x,u,Xu)   \in L^{1}_{loc}(\Omega-\Sigma)$,
then $u \in \mathnormal{W}^{1,p}_{X,loc}(\Omega)$ is a weak solution of \eqref{equa} in $\Omega$.
\end{Thm}
\begin{Rem}
We can see if \eqref{cond_prime} holds, then we can deduce that
$B(x,u,Xu)\in L^{1}_{loc}(\Omega-\Sigma)$ from $u \in \mathnormal{W}^{1,p}_{X,loc}(\Omega-\Sigma)$. It implies  that $B(x,u,Xu)\in L^{1}_{loc}(\Omega-\Sigma)$   represents a weaker structure condition.
\end{Rem}

\begin{corollary}
Let $U$ be an connected open domain of $\mathbb{R}^{n}$, $\Omega\subset U$ be  a bounded open domain, and  $\Sigma \subset U$ be a compact set with $cap_p(\Sigma)=0$.  Assume that  \eqref{cond}, \eqref{func} are satisfied,   and that
\begin{equation}\label{thmcond1}
u \cdot B(x,u,Xu)  \leq (1-\theta)|Xu|^p+|u|\{c(x)|Xu|^{p-1}+d(x)|u|^{p-1}+f(x)\}
\end{equation}
holds for $a.e. \; x \in \Omega-\Sigma$, where $ \theta \in (0,1]$.   If $u\in \mathnormal{W}^{1,p}_{X,loc}(\Omega-\Sigma) \cap L^{\infty}_{loc}(\Omega)$  is a weak solution of \eqref{equa} in $\Omega-\Sigma$ and  $B(x,u,Xu)   \in L^{1}_{loc}(\Omega-\Sigma)$,
then $u \in \mathnormal{W}^{1,p}_{X,loc}(\Omega)$ is a weak solution of \eqref{equa} in  $\Omega$.
\end{corollary}
\begin{Thm}\label{Thmholderctn}
Let $U$ be an connected open domain of $\mathbb{R}^{n}$, $\Omega\subset U$ be  a bounded open domain, and  $\Sigma \subset U$ be a compact set with $cap_s(\Sigma)=0$ for some $s\in [p, \tilde{v}]$,  where $\tilde{v}$ is the generalized M\'etivier index of $\Omega$.  Assume that  \eqref{cond}, \eqref{cond_prime} and \eqref{func} are satisfied. If $u\in \mathnormal{W}^{1,p}_{X,loc}(\Omega-\Sigma)$ is a bounded weak solution of \eqref{equa} in $\Omega-\Sigma$,
then $u \in \mathnormal{W}^{1,p}_{X,loc}(\Omega)$ is a bounded weak solution of \eqref{equa} in  $\Omega.$
\end{Thm}

\begin{Thm}\label{Thm5.1}
 Let $U$ be an connected open domain of $\mathbb{R}^{n}$, $\Omega\subset U$ be  a bounded open domain,  and  $u\in \mathnormal{W}^{1,p}_{X,loc}(\Omega)$ be a non-negative bounded weak solution of equation \eqref{equa}. If conditions  \eqref{cond}, \eqref{cond_prime} and \eqref{func} hold, then there exist positive constants $C$,  $\rho_{\Omega}$  such that for any $B_{R}$ with $B_{4R}\subset\subset \Omega$ satisfying $R\leq \frac{\rho_{\Omega}}{4},$

when $p<\tilde{v}$, for some $\gamma_0>0$
\begin{equation}\label{571}
 \sup\limits_{B_{R}} u \leq C\left({\frac{|B_{2R}|}{R^{\tilde{v}}}}\right)^{\frac{2}{\gamma_0}} \left(\inf\limits_{B_{R}} u +\tilde k(R)\right),
\end{equation}
in additional,  if $\Omega$ is  equiregular connected set, we have
\begin{equation}\label{eq:45}
 \sup\limits_{B_{R}} u \leq C \left(\inf\limits_{B_{R}} u +\tilde k(R)\right);
\end{equation}

 when $p=\tilde{v}$, for some $\gamma_0>0,$  for any $\delta>0$,
 \begin{equation}
 \sup\limits_{B_{R}} u \leq C\left({\frac{|B_{2R}|}{R^{\tilde{v}+\delta}}}\right)^{\frac{2}{\gamma_0}} \left(\inf\limits_{B_{R}} u +\tilde k(R)\right),
\end{equation}
Here $$C=C(n,p,\varepsilon ,a,c_0,R,||d||_{\Omega},||c||_{\Omega},||b||_{\Omega},||u||_{L^\infty (\Omega)})$$
and $$\tilde k(R)=(||e||_{B_{4R}}+|B_{R}|^{\frac{\varepsilon}{2\tilde v}}||f||_{B_{4R}})^{\frac{1}{p-1}}+(|B_{R}|^{\frac{\varepsilon}{2\tilde v}}||g||_{B_{4R}})^{\frac{1}{p}},$$
where $||c||,||b||,||d||,||e||,||f||$ and $||g||$ are as defined in the corresponding norm of \eqref{func}.
\end{Thm}

\begin{Thm}\label{Thmholderctn1}
 Let $U$ be an connected open domain of $\mathbb{R}^{n}$, $\Omega\subset U$ be  a bounded open domain,  and  $u\in \mathnormal{W}^{1,p}_{X,loc}(\Omega)$ be a bounded weak solution of equation \eqref{equa}. Assume $p<\tilde{v}$, conditions  \eqref{cond}, \eqref{cond_prime} and \eqref{func} hold,  and $\Omega$ is equiregular connected set. Additionally, if $ b,e \in L_{\frac{\tilde{v}}{p-1-\varepsilon}} (\Omega)$,  then for any compact subset $\Omega_0 \subset \Omega$, $u$ is H\"older continuous in $\Omega_0.$
\end{Thm}
 For the Grushin type opterators in $\mathbb{R}^n$ as \eqref{eq:34},
 we consider the equation \begin{equation}\label{eq:36}P_{k'}(u)=f(x,z). \end{equation}


Recalling  \cite{BFI},   W. Bauer,   K. Furutani and   C. Iwasaki  examined a class of Grushin type operators as \eqref{eq:34}. Based on the use of polar coordinates and certain changes of variables they  reduced $P_{k'}$ to a simpler structure. After that, they applied the theory
of Bessel and modified Bessel functions together with Weber's second exponential integral
to derive an exact form of the fundamental solution of the  simpler structure. Finally, they  obtained  the  explicit expression of fundamental solution of $P_{k'}$ by returning back to
Cartesian coordinates.

\begin{lemma}[\cite{BFI}, Theorem 7.5]\label{Thm7.5}
Suppose that $(x,z)\neq(x',z')$, $q\in \mathbb{N}$ and set
\begin{equation}\label{eq:38}
\gamma=|x|^{k'+1}|x'|^{k'+1}, \ A=\frac{|x|^{2(k'+1)}+|x'|^{2(k'+1)}+(k'+1)^2|z-z'|^2}{2},\ \tau=\frac{\langle x,x'\rangle}{|x||x'|}.\end{equation}
\textup{(I)} If $l=2q$ is even, then
$$ (P_{k'})^{-1}\left((x,z),(x',z')\right)=\frac{k+1}{4\pi}F_{q,k'+1,h}(A,\gamma,\tau).$$
\textup{(II)} If $l=2q-1$ is odd, then
$$ (P_{k'})^{-1}\left((x,z),(x',z')\right)=\frac{\sqrt{2}}{4\pi}\int_{A}^{\infty}F_{q,k'+1,h}(v,\gamma,\tau)\frac{1}{\sqrt{v-A}}dv.$$
Here,  for $q\geq2$, $\displaystyle{F_{q,k'+1,h}(v,\gamma,\tau)=\left(\frac{-(k'+1)}{2\pi}\frac{\partial}{\partial v}\right)^{q-1}F_{1,k'+1,h}(v,\gamma,\tau)}$  and $F_{1,k'+1,h}(v,\gamma,\tau)$ is given as follows
$$F_{1,k'+1,h}(v,\gamma,\tau)=\frac{\Gamma(h/2)}{2{\pi}^{h/2}}\times \frac{(v+\sqrt{v^2-\gamma^2})^{1/(k'+1)}-(v-\sqrt{v^2-\gamma^2})^{1/(k'+1)}}{\sqrt{v^2-\gamma^2}\{(v+\sqrt{v^2-\gamma^2})^{1/(k'+1)}+(v-\sqrt{v^2-\gamma^2})^{1/(k'+1)}-2\gamma^{1/(k'+1)}\tau\}^{h/2}},$$
where $ \Gamma(h/2)/(2{\pi}^{h/2})=|S^{h-1}|^{-1}$ is the inverse of the volume of the $(h-1)$-dimensional Euclidean unit sphere.

\end{lemma}

\begin{lemma}[\cite{BFI}, Corollary 7.6]
With the notation in Lemma \ref{Thm7.5} we have for $\gamma=0$, any $l\in \mathbb{N}$
$$ (P_{k'})^{-1}\left((x,z),(x',z')\right)=\frac{(k'+1)^{l-1}\Gamma\left(\frac{l}{2}+\frac{h-2}{2(k'+1)}\right)\Gamma(h/2)}{4\pi^{l/2+h/2}\Gamma\left(1+\frac{h-2}{2(k'+1)}\right)(2A)^{\frac{l}{2}+\frac{h-2}{2(k'+1)}}}.$$
\end{lemma}
\begin{Rem}For $(x',z')\in \Omega$, if $x'=0$, similar to the classical Laplace equation, we conclude that $(0,z')$
 is a removable singularity when $u$ is a higher-order infinitesimal of a fundamental solution of $P_{k'}$, as shown in \eqref{eq:37}. However, if $x'\neq 0$, we cannot derive the above precise result. This is because the cross terms between $x$ and $z-z'$, as well as between $x'$ and $z-z'$,  introduce difficulties into the analysis. When $x'\neq 0$, the question of at least to what order of higher-order infinitesimal of the fundamental solution $u$ must satisfy for $(x',z')$ to be a removable singularity  is still  an open problem.\end{Rem}


\begin{corollary}\label{thm4.1}
Let $\Omega$ be a bounded open domain of $ \mathbb{R}^n$, and  $\Sigma \subset \mathbb{R}^n$ be a compact set with $cap_s(\Sigma)=0$ for some $s\in [2, \tilde{v}]$, where $\tilde{v}$ is the generalized M\'etivier index of $\Omega$. Assume  $f  \in  L_{\frac{\tilde{v}}{2-\varepsilon}} (\Omega)$ for some $\varepsilon\in (0,1)$.   For any fixed $\delta>0$, if $u\in \mathnormal{W}^{1,2}_{X,loc}(\Omega-\Sigma) \cap L_{\frac{s}{s-2}(1+\delta),loc}(\Omega)$ be a weak solution of \eqref{eq:28} in $\Omega-\Sigma$,
then $u \in \mathnormal{W}^{1,2}_{X,loc}(\Omega)$ is a weak solution of \eqref{eq:28} in  $\Omega$.  In particular, if $s=2$ and $u\in L^{\infty}_{loc}(\Omega)$, the conclusion still hold.
\end{corollary}
\begin{Thm}\label{thm4.2}
Let $\Omega$ be a bounded open domain of $ \mathbb{R}^n$, and  $\{(0,z')\}\subset \Omega$.  Let $u\in \mathnormal{W}^{1,2}_{X,loc}(\Omega-\{(0,z')\})$ be a weak solution of \eqref{eq:28} in $\Omega-\{(0,z')\}$. Assume $f  \in  L_{\frac{h+l(k'+1)}{2-\varepsilon}} (\Omega)$ for some $\varepsilon\in (0,1)$.
  If for any fixed $\tilde{\delta}> \frac{lk'}{2(k'+1)}\frac{h+l(k'+1)-2}{h+l(k'+1)}$
\begin{equation}\label{eq:37}
u (x,z)=O\left(\frac{1}{(|x|^{2(k'+1)}+(k'+1)^2|z-z'|^2)^{\frac{l}{2}+\frac{{h-2}}{{2}(k'+1)}-\tilde{\delta}}}\right), \ \ \text{as} \  |\zeta|, |\varsigma|\to 0,
\end{equation}
where  $\zeta, \varsigma$ denotes the component of distance between  $(x,z)$ and $(0,z')$, respectively,
then $u \in \mathnormal{W}^{1,2}_{X,loc}(\Omega)$  is a weak solution of \eqref{eq:28} in $\Omega$.
\end{Thm}
\begin{Rem}
It is worth noting that $\tilde{v}(0,z')=h+l(k'+1)$.
\end{Rem}
\begin{Rem} It is well-known that  for classical  Laplace equation
\begin{equation}\label{eq:47}-\Delta u(y)=0 \ \ \ \ y\in\mathbb{R}^n\setminus \{y_0\} , \end{equation}
if $y_0$ is removable for the solutions $u$ to \eqref{eq:47}
  if only if  $$u(y)=\begin{cases}
  o\left(\ln|y-y_0|\right)(y\to y_0)\ \ \ & \textup{for}\  n=2,\\
o\left(|y-y_0|^{2-n}\right)(y\to y_0)  \ & \textup{for}\   \ n\geq 3.
\end{cases}$$
It is worth noting that when $k'=0$, we can see that the condition in Theorem \ref{thm4.2} becomes
 $\tilde{\delta}>0$, which is consistent with the result for the above classical Laplace equation.
\end{Rem}

\begin{corollary}\label{Thm1.7}
Let $\Omega$ be a bounded open domain of $ \mathbb{R}^n$,  and  $u\in \mathnormal{W}^{1,2}_{X,loc}(\Omega)$ be a non-negative bounded weak solution of equation \eqref{eq:28}. If $f  \in  L_{\frac{\tilde{v}}{2-\varepsilon}} (\Omega)$ for some $\varepsilon\in (0,1)$, then there exist positive constants $C$,  $\rho_{\Omega}$  such that for any $B_{R}$ with $B_{4R}\subset\subset \Omega$ satisfying $R\leq \frac{\rho_{\Omega}}{4},$
\begin{equation}\label{571}
 \sup\limits_{B_{R}} u \leq C\left({\frac{|B_{2R}|}{R^{\tilde{v}}}}\right)^{\frac{2}{\gamma_0}} \left(\inf\limits_{B_{R}} u +\tilde k(R)\right),\  \textup{ for some }\ \gamma_0>0,
\end{equation}in additional,  if $\Omega$ is  equiregular connected set, we have
\begin{equation}\label{eq:46}
 \sup\limits_{B_{R}} u \leq C \left(\inf\limits_{B_{R}} u +\tilde k(R)\right),\  \textup{ for some }\ \gamma_0>0.
\end{equation}
Here $$C=C(n,\varepsilon, R,||u||_{L^\infty (\Omega)})$$
and $$\tilde k(R)=|B_{R}|^{\frac{\varepsilon}{2\tilde v}}||f||_{L_{\frac{\tilde{v}}{2-\varepsilon}} (B_{4R})}.$$
\end{corollary}

\begin{corollary}\label{Thm1.8}Let $\Omega$ be a bounded open domain of $ \mathbb{R}^n$,  and  $u\in \mathnormal{W}^{1,2}_{X,loc}(\Omega)$ be a bounded weak solution of equation \eqref{eq:28}. Assume $\Omega$ is equiregular connected set  and $f  \in  L_{\frac{\tilde{v}}{2-\varepsilon}} (\Omega)$ for some $\varepsilon\in (0,1)$.
   Then for any compact subset $\Omega_0 \subset \Omega$, $u$ is H\"older continuous in $\Omega_0.$
\end{corollary}

Here we highlight the contributions of this paper.  Firstly,  we  extended the results  in \cite{MM1} to degenerate subelliptic equations, obtaining  the  removable singularities  and  Harnack inequality for  solutions of  \eqref{equa}. Secondly, relative to the work in \cite{LDN, LC1},  we focus on general H\"ormander vector fields.  This poses great challenges for us, and we have to handle the related problem without assumption $\text{(H)}$.  Finally, in terms of details, it is worth mentioning that we   weaken the integrability of coefficients in structure conditions \eqref{func} by a sharp Sobolev inequalities, which is mainly reflected in $\tilde{v}\leq Q$.


 This paper is organized as follows. In Section 2, we mainly give some preliminaries.  In Section 3, we prove Theorem \ref{Thmremovsing}  and  \ref{Thmholderctn} to study removable singularities  of  weak  solutions for \eqref{equa}. In Section 4, we provide  the Harnack inequality and H\"older  continuity  for non-negative bounded weak solutions of  \eqref{equa} in  Theorem \ref{Thm5.1} and  Theorem \ref{Thmholderctn1}. In Section 5, we give an application  to a  higher step Grushin type operator and prove the Theorem \ref{thm4.2}.


\section{Preliminaries}
In this section, we will give some definitions and lemmas as a preparation, such as Sobolev spaces associated with the H\"ormander vector fields $X$, subunit metric, chain rule  and other auxiliary  knowledge.
\begin{definition}[Weak derivatives]\label{def2.1}
Let $U\subset\mathbb{R}^{n}$ be an open set and let $Y$ be a smooth vector
field in $U$. We say a given $u\in L_{loc}^{1}(U)$ is differentiable in weak sense with respect to $Y$ if there
exists a $g\in L_{loc}^{1}(U)$ such that for every $\varphi\in C_{0}^{\infty}(U)$,
$$\int_{U}g(x) \varphi dx=\int_{U}u(x) Y^{*}\varphi dx$$
where the transpose operator $Y$ is defined as follows: if
 $$Y\varphi(x)=\sum\limits_{j=1}^{n}h_{j}(x)\partial_{x_j}\varphi(x),$$
  then $$Y^{*}\varphi(x)=-\sum\limits_{j=1}^{n}\partial_{x_j}
  \left(h_{j}(x)\varphi(x)\right). $$ In this case we will write $g=Yu$.

\end{definition}
\begin{definition}[\cite{GSP}]\label{Def2.2}
Let $1\leq j_i\leq m$ and $J=(j_1,\cdots, j_l)$ denotes a multi-index with length $|J|=l$. We adopt the notation $X^{J}=X_{j_1}X_{j_2}\cdots X_{j_{l-1}}X_{j_l}$ for $|J|=l$, and $X^{J}=id$ for $|J|=0$.
For any $k\in\mathbb{N}^{+}$ and $p\geq 1$, we define the function space
$$\mathnormal{W}^{k,p}_{X}(\Omega)=\{u\in L^p(\Omega)| \ X^Ju\in L^p(\Omega),\   \forall J=(j_1,...j_s)\ \text{with}\ |J|\leq k\}, $$
and set the norm
$$||u||^p_{\mathnormal{W}^{k,p}_{X}(\Omega)}=\sum\limits_{|J|\leq k}||X^Ju||^p_{L^p(\Omega)}. $$  Here, $X^{J}$ is the weak derivatives.
Furthermore we define  $\mathnormal{W}^{1,p}_{X,0}(\Omega)$ as the closure of $C_{0}^{\infty}(\Omega)$ in $\mathnormal{W}^{1,p}_{X}(U)$.
\end{definition}

\begin{definition}
Assume that \eqref{cond},  \eqref{func} are satisfied for equation \eqref{equa}.
 Then a function $u \in \mathnormal{W}^{1,p}_{X,loc}(\Omega) $ is called a weak solution of \eqref{equa} in $\Omega$ if $B(x,u,Xu) \in L^1_{loc} (\Omega)$, and
\begin{equation}\label{weso}
 \int_{\Omega}A(x,u,{X}u)\cdot X\varphi-B(x,u,{X}u)\varphi dx=0
\end{equation}
holds for all test function \;$\varphi \in \mathnormal{C}^{\infty}_{0}(\Omega)$.
\end{definition}

\begin{definition}[(subunit metric)\cite{GSP}, Definition 2.1] \label{ccd}
 For any $x,y\in U$ and $\delta >0$, let $C(x,y,\delta)$ be the collection of absolutely continuous mapping $\varphi :[0,1]\to U$, such that $\varphi(0)=x,\varphi(1)=y$ and $$\varphi'(t) = \sum\limits_{i=1}^{m} a_i (t)(X_i)_{ \varphi(t)} $$ with $\sum\limits_{k=1}^{m} |a_k (t)|^2 \leq {\delta}^2$ for a.e. $t\in [0,1]$.
 The subunit metric $d(x,y)$ is defined by
$$ d(x,y):= \inf\{\delta >0 | \ \exists \  \varphi \in C(x,y,\delta) \ \text{with}\  \varphi(0)=x, \ \varphi(1)=y \}.$$
We also  denote the subunit ball  by $$B(x,R):=\{y\in U |d(x,y)<R\}.$$  When we don't emphasize $x$, we simply write $B(x,R)$ as $B_{R}$.

\end{definition}

Now, we introduce the following properties about generalized Sobolev spaces. For details, refer to  \cite{GSP}.
\begin{lemma}[\cite{GSP}, Proposition 2.8] \label{chain}
Let $U_1$ be an open subset of U. Suppose that $F\in C^1(\mathbb{R})$ with $F'\in L^{\infty}(\mathbb{R})$. Then for any $u\in\mathnormal{W}^{1,p}_{X}(U_1)$ with $ p\geq1$,we have
$$X_j(F(u))=F'(u)X_j u  \qquad  in \ \mathscr{D}'(U_1) \qquad  \text{for} \ j=1,...m.$$
Moreover,\\
(1)if $F(0)=0$, then $F(u)\in\mathnormal{W}^{1,p}_{X}(U_1); $ \\
(2)if $F(0)=0$ and $u\in\mathnormal{W}^{1,p}_{X,0}(U_1)$, then $F(u)\in\mathnormal{W}^{1,p}_{X,0}(U_1).$
\end{lemma}

\begin{lemma}[\cite{GSP} Proposition 2.9] \label{chain2}
Let $U_1$ be an open subset of U. For any $u\in\mathnormal{W}^{1,p}_{X}(U_1)$ and any $c \in \mathbb{R}$, we have
$$ X_j (u-c)^{+}=H(u-c)X_j u \quad and \quad X_j (u-c)^{-}=-H(c-u)X_j u \quad \text{in}\   \mathscr{D}'(U_1) ,$$
where $H(x)=\chi _{\{x\in \mathbb{R}|x>0\}}(x)$ and $\chi_{E}$ denotes the indicator function of E.
Furthermore,

(1) if $c\geq 0$,we have $(u-c)^{+}, (u-c)^{-} \in\mathnormal{W}^{1,p}_{X}(U_1); $

(2) if $c\geq 0$ and $u\in\mathnormal{W}^{1,p}_{X,0}(U_1)$,then $(u-c)^{+}, (u-c)^{-} \in\mathnormal{W}^{1,p}_{X,0}(U_1).$\\
Here $(u-c)^{+}=\max\{u-c,0\}$, $(u-c)^{-}=-\min \{u-c, 0\}$.
\end{lemma}

\begin{lemma} \label{chain3}
Let $U_1$ be an open subset of U, and  F be a piecewise smooth function on $\mathbb{R}$ with $F'\in L^{\infty}(\mathbb{R})$. Then  if $u\in\mathnormal{W}^{1,p}_{X}(U_1)$,  letting $\mathscr {L}$ denote the set of corner points of $F$,  we have
$$X_j (F(u))=\left\{
\begin{aligned}
&F'(u)X_j u, \qquad &if\  u\notin \mathscr {L},\\
&0,\qquad &if\  u\in \mathscr {L}.\\
\end{aligned}
\right. $$
Furthermore, if $F\circ u\in L^{p}(U_1)$,  we have $F\circ u\in\mathnormal{W}^{1,p}_{X}(U_1).$

\end{lemma}
\begin{proof}
Let us assume  $c$ is a corner point. Let $F_1, F_2\in C^{1}(\mathbb{R})$ satisfying $F_1', F_2'\in L^{\infty}(\mathbb{R})$ and
$$F_1=F, \ \ \text{if} \ u \geq c;\ \ \ F_2=F, \ \ \text{if} \ u \leq c.$$

Through Lemma \ref{chain}, we have
$$X_j(F_1 (u))=F_1'(u) X_ju\ \  \text{and}\ \  X_j(F_2(u))=F_2'(u) X_ju. $$
Since
$$
F(u)=
 \begin{cases}
F_1((u-c)^{+}+c)  \ \ \ & u\geq c \\
F_2(-(u-c)^{-}+c)  \ &  u\leq c,
\end{cases}
$$
combing  Lemma \ref{chain2}, we have
$$\begin{aligned}
X_j (F(u))=&
 \begin{cases}
F_1'((u-c)^{+}+c) X_j(u-c)^{+} \ \ \ & u\geq c \\
F_2'(-(u-c)^{-}+c) X_j(-(u-c)^{-})  \ &  u\leq c,
\end{cases}
\\=& \begin{cases}
F'(u)X_j u\ \ \ & u\neq c \\
0  \ &  u= c.
\end{cases}
\end{aligned}$$

 For another corner   point $c_1$, we repeat the  analysis  for $F_1((u-c)^{+}+c)$ or $F_2(-(u-c)^{-}+c)$.
    Keep  on the above steps  and  we conclude.
\end{proof}
\begin{lemma}[\cite{GSP}, Theorem 1.1]\label{sobolev}
Let $X=(X_1,X_2,...X_m)$ satisfy condition \textup{(H)}. Then, for any bounded open subset $\Omega \subset\subset U$ and any positive number $p\geq  1$, there exist a positive constant $C>0$ such that\\
(1)If $kp<\tilde{v}$ and $\frac{1}{q}=\frac{1}{p}-\frac{k}{\tilde{v}}$, we have
$$||u||_{L^q(\Omega)}\leq C\sum\limits_{|J|=k}||X^Ju||_{L^p(\Omega)}  \qquad  \forall u\in \mathnormal{W}^{k,p}_{X,0}(\Omega); $$
(2)If $kp=\tilde{v}$ and $1\leq q<\infty$, we also have
$$||u||_{L^q(\Omega)}\leq C\sum\limits_{|J|=k}||X^Ju||_{L^p(\Omega)}  \qquad  \forall u\in \mathnormal{W}^{k,p}_{X,0}(\Omega). $$
Here, $\tilde{v}$ is the generalized M\'etivier index of $\Omega$.
\end{lemma}

\begin{lemma}\label{caps} For any bounded open  set $\Omega \subset\subset U$,
if $\Sigma \subset U $ be a compact set with $cap_s(\Sigma)=0$, where $1\leq s\leq \tilde{v}$, then there exists a sequence of functions $\bar{\eta}_v \in \mathnormal{C}^{\infty}(U)$ with the properties:\\
(i) $\bar{\eta}_v$ vanish in a neighborhood of  $\Sigma,$ \\
(ii) $0 \leq \bar{\eta}_v \leq 1,$\\
(iii) $\bar{\eta}_v \rightarrow 1 \; a.e. \text{ in} \; U\; \text{ as} \; v \rightarrow \infty ,$\\
(iv) $\int_{U}|X \bar{\eta}_v|^s dx \rightarrow 0 \;  \text{ as}  \; v \rightarrow \infty.$
\end{lemma}
\begin{proof}
Since $cap_s(\Sigma)=0$, there is a bounded open $U_1$ such that $cap_{s}(\Sigma, U_1 )=0$. Form the Definition \ref{definition1}, we can assume $\Omega\subset U_1$. Then there is a sequence of functions $\psi_{v}\in \mathnormal{C}^{\infty}_{0}(U_1 )$ such that  $$ \int_{U} |X{\psi}_v|^{s}dx\leq \frac{1}{v},$$
and
  $2{\psi}_v\geq 1$  in some neighborhood of $\Sigma$, so we define
$${\bar{\psi}}_v=\left\{
\begin{aligned}
&0 & if \ {\psi}_v <0 ,\\
&2{\psi}_v \qquad &if \  0\leq2{\psi}_v\leq 1,\\
&1 &if \ 2{\psi}_v >1.
\end{aligned}
\right.$$

  It's obvious that  ${\bar{\psi}}_v$ is  Lipschitz continuous and
  $$ \int_{U} |X\bar{{\psi}}_v|^{s}dx\leq 2^s\int_{U} |X{\psi}_v|^{s}dx\leq \frac{ 2^s}{v}.$$

  By Lemma \ref{sobolev}, we obtain
  $$  \int_{U_1} |{\bar{\psi}}_v|^{s*}dx\leq   \int_{U_1}|X {\bar{\psi}}_v|^{s}dx \leq \frac{ C\cdot2^s}{v},$$
  with
   $$s^{\ast}=\left\{\begin{aligned}&\frac{s\tilde{v}(U_1)}{\tilde{v}(U_1)-s},\quad &if \ s<\tilde{v}(U_1) ,\\
&2s,\quad &if \ s=\tilde{v}(U_1),
\end{aligned}
\right.$$
 which implies ${\bar{\psi}}_v=0$ a.e. in $U_1$ when $v\to \infty $. Then $\bar{{\psi}}_v=0$ a.e. in $U$ when $v\to \infty $.
 Let $\bar{\psi}_{v,j}$ is the mollification of   $\bar{\psi}_v$,   then  $\bar{\psi}_{v,j}\in C^{\infty}(U)$, $0 \leq\bar{\psi}_{v,j}\leq 1$,  and $\bar{\psi}_{v,j} \rightarrow 0$  a.e. in  $U$. Furthermore,  it is shown in [\cite{NGDMN}, p. 1136] that
 $$\int_{U}|X \bar{{\psi}}_{v,j}|^s dx \rightarrow \int_{U}|X \bar{{\psi}}_{v}|^s dx \rightarrow 0. $$

  Set $\tilde{{\psi}}_{v,j}=1-\bar{\psi}_{v,j}$. Since $\tilde{v}\leq \tilde{v}(U_1)$, it's easy see that
$\tilde{{\psi}}_{v,j}$ vanish in a neighborhood of  $\Sigma,$
$0 \leq \tilde{{\psi}}_{v,j}\leq 1,$ and as  $v \rightarrow \infty$,
 $\tilde{{\psi}}_{v,j}\rightarrow 1$  a.e. in  $U$ and
 $\int_{U}|X \tilde{{\psi}}_{v,j}|^s dx \rightarrow 0 $ for any $s\in [1, \tilde{v}]$.

\end{proof}

\begin{lemma}[\cite{MM} Lemma 1]\label{inez}
Let $\alpha$ be a positive exponent, and let $\alpha_i$, $\beta_i (i=1,...N)$ be real numbers with the properties $0\leq \alpha_i< \infty $ and $0\leq \beta_i< \alpha .$ If z is a nonnegative real number satisfying the inequality
$$z^\alpha \leq \sum\limits_{i=1}^N \alpha_i z^{\beta_i}, $$
then
$$z \leq C \sum\limits_{i=1}^N { \alpha_i}^{(\alpha-\beta_i)^{-1}} .$$
where C depends only on N, $\alpha$, and $\beta_i$.
\end{lemma}

%

\section{ Removable singularity }
In this section, we mainly studies    removable singularities for weak solutions of \eqref{equa}  by  proving the Theorem  \ref{Thmremovsing} and  Theorem  \ref{Thmholderctn}.

\begin{proof}[\textbf{proof of Theorem \ref{Thmremovsing}}]
It suffices to prove that for any ball $B_{2R}\subset\subset \Omega $, the function $ u\in \mathnormal{W}^{1,p}_{X}(B_{R})$ and $u$ is a weak solution of \eqref{equa} in $B_{R}$. Next, we will discuss in two steps.

Step1:  we will show that $||u||_{L^p(B_R-\Sigma)}$ and $||X u||_{L^p(B_R-\Sigma)}$ are finite.\\
Let $$k=1+\left(||e||_{B_{2R}}+||f||_{B_{2R}}\right)^{\frac{1}{p-1}}+||g||_{B_{2R}}^{\frac{1}{p}},$$ and set
$$\bar{u}=|u|+k,\quad \bar{b}(x)=b(x)+k^{1-p}e,\quad \bar{d}(x)=d(x)+k^{1-p}f(x)+k^{-p}g(x), $$ where $||e||,||f||$ and $||g||$ are as defined in the corresponding norm of \eqref{func}.
We can see that  $||\bar{b}(x)||_{B_{2R}}\leq ||b(x)||_{B_{2R}}+1 ,$ $||\bar{d}(x)||_{B_{2R}}\leq ||d(x)||_{B_{2R}}+2,$
 $|u|\leq \bar{u},$ and
\begin{equation}\label{thmesti1}
\left\{
\begin{aligned}
&b(x)|u|^{p-1}+e(x) \leq  \bar {b}(x)\bar {u}^{p-1}, \\
&d(x)|u|^p+g(x) \leq  \bar {d}(x)\bar {u}^p , \\
&d(x)|u|^{p-1}+f(x) \leq  \bar {d}(x)\bar {u}^{p-1}.
\end{aligned}
\right.
\end{equation}
Moreover, for any $l>k$, we set
\begin{equation}\label{eq:11}  \bar{u}^{(l)}_k(x) =\begin{cases}
l  \ \ \ & if \ \bar{u}\geq l,\\
\bar{u}   \ &if \ k<\bar{u}<l,\\
k \ &if \ \bar{u}\leq k,
\end{cases}
 \end{equation}
 and
\begin{equation}\label{eq:12}
\bar{u}_k(x)=\bar{u}^{(\infty)}_k(x).\end{equation}
Besides, we set  \begin{equation}\label{eq:133}\displaystyle q_0=\frac{p-\theta}{p}(1+\delta)\end{equation}and assume without loss of generality that  $q_0\leq 1.$
 For any number $q$ satisfying $0<q_0\leq q\leq 1$,   we set \begin{equation}\label{eq:13}t_0=p(q_0-1),\ t=p(q-1).\end{equation} Then we have  \begin{equation}\label{eq:14}-p\leq t_0\leq t\leq 0, \  t_0+\theta >0.\end{equation}

Let
\begin{equation}\label{thmtestf}
\tilde{\phi}=(\eta\bar{\eta})^p u( \bar{u}_k)^{t_0}(\bar{u}^{(l)}_k)^{t-t_0}=:(\eta\bar{\eta})^p u \psi,
\end{equation}
where $\eta\in \mathnormal{C}^{\infty}_0(B_{2R}),$ $\bar{\eta}\in \mathnormal{C}^{\infty} ({\mathbb {R}}^n)$ satisfying $0\leq \eta,\bar{\eta}\leq 1$ and $\bar{\eta}$ vanishes in a neighborhood of $\Sigma$.
 Since $\tilde{\phi}(u)$ is  piecewise
smooth  about $u$ and $\tilde{\phi}'\in L^{\infty}(\mathbb{R})$, by Lemma \ref{chain3} we get that $\tilde{\phi} \in \mathnormal{W}^{1,p}_{X,0}(B_{2R}-\Sigma)$ and  the following equation
\begin{equation}\label{eq:1}
X\tilde{\phi}=p(\eta\bar{\eta})^{p-1}X(\eta\bar{\eta})u \psi +(\eta\bar{\eta})^p X u \psi
+(\eta\bar{\eta})^p |u|X u \psi (\bar{u}_k)^{-1}\{t \chi _{k,l}+t_0 \chi_l\}.\end{equation}
where $\chi_l $  and  $\chi _{k,l}$ denote the characteristic functions of the sets $\{x\in B_{2R}:l<\bar{u}\}$ and $\{x\in B_{2R}:k<\bar{u}<l\}$, respectively.
For any $i$ we define
\begin{equation}\label{eq:111}
{ \tilde\phi}^{(i)}=\left\{
\begin{aligned}
& \tilde\phi & if \ \tilde\phi \leq i ,\\
& \frac{ \tilde\phi}{| \tilde\phi|}i   &if \  \tilde\phi \geq i,
\end{aligned}
\right.\end{equation}
and we have ${\tilde\phi}^{(i)} \in \mathnormal{W}^{1,p}_{X,0}\cap L^{\infty}(B_{2R}-\Sigma).$

Based on the fact that u is a weak solution of \eqref{equa} in $\Omega-\Sigma$ and $B(x,u,Xu)\in L^{1}_{loc}(\Omega-\Sigma)$, letting $\tilde{\phi}^{i}$ be test function, we have
$$
 \int_{B_{2R}-\Sigma}A(x,u,X u)\cdot X\tilde{\phi}^{i} dx= \int_{B_{2R}-\Sigma}B(x,u,Xu)\tilde{\phi}^{i}dx\leq  \int_{B_{2R}-\Sigma}|B(x,u,Xu)\tilde{\phi}|dx.
$$
Let $i\to \infty$, we get
$$\int_{B_{2R}-\Sigma}A(x,u,X u)\cdot X\tilde{\phi}dx \leq  \int_{B_{2R}-\Sigma}|B(x,u,Xu)\tilde{\phi}|dx. $$
Combining  \eqref{thmcond1} and  \eqref{eq:1}  in sequence,  we have
$$\begin{aligned}&\int_{B_{2R}-\Sigma}  (\eta\bar{\eta})^p \psi Xu A(x,u,Xu)\left(1+|u|(\bar{u}_k)^{-1}\{t \chi _{k,l}+t_0 \chi_l\}\right)dx \\
\leq &\int_{B_{2R}-\Sigma} (\eta\bar{\eta})^p\psi \left((1-\theta)|Xu|^p+|u|(c(x)|Xu|^{p-1}+{d}(x)|{u}|^{p-1}+f(x))\right)\\
    &+p(\eta\bar{\eta})^{p-1}X(\eta\bar{\eta})|u| \psi|A(x,u,Xu)|dx.
    \end{aligned}$$

    By \eqref{cond}, we have
 $$\begin{aligned}&\int_{B_{2R}-\Sigma}  (\eta\bar{\eta})^p \psi \left(1+|u|(\bar{u}_k)^{-1}\{t \chi _{k,l}+t_0 \chi_l\}\right)\left(|{X}u|^p-d(x)|u|^p-g(x)\right)dx \\
\leq &\int_{B_{2R}-\Sigma} (\eta\bar{\eta})^p\psi \left((1-\theta)|Xu|^p+|u|(c(x)|Xu|^{p-1}+{d}(x)|{u}|^{p-1}+f(x))\right)\\
    &+p(\eta\bar{\eta})^{p-1}X(\eta\bar{\eta})|u| \psi \left(a|Xu|^{p-1}+{b}(x)|{u}|^{p-1}+e(x)\right)dx.
    \end{aligned}$$

And by \eqref{thmesti1} and  \eqref{eq:12}, we get

$$\begin{aligned}&\int_{B_{2R}-\Sigma}  (\eta\bar{\eta})^p \psi \left(\theta+|u|(\bar{u}_k)^{-1}\{t \chi _{k,l}+t_0 \chi_l\}\right)|{X}u|^pdx \\
\leq &\int_{B_{2R}-\Sigma} (\eta\bar{\eta})^p\psi \bigg(|u|(c(x)|Xu|^{p-1}+\bar{d}(x)\bar{u}^{p-1})+ \left(1+|u|(\bar{u}_k)^{-1}\{t \chi _{k,l}+t_0 \chi_l\}\right)\bar{d}(x)\bar{u}^p\bigg)\\
    &+p(\eta\bar{\eta})^{p-1}X(\eta\bar{\eta})|u| \psi \left(a|Xu|^{p-1}+\bar{b}(x)\bar{u}^{p-1}\right)dx.
    \end{aligned}$$

 By \eqref{eq:11}, \eqref{eq:13} and  \eqref{eq:14}, we  note
 $t_0\leq t\leq 0 $, and $ t_0+\theta> 0$.  Combing   \eqref{cond} and \eqref{thmesti1} we obtain that
\begin{equation}\label{ine1}
\begin{aligned}
    &\int_{B_{2R}-\Sigma} (t_0+\theta)(\eta\bar{\eta})^p \psi |{X}u|^p dx
  \leq   \int_{B_{2R}-\Sigma} |c(x)|(\eta\bar{\eta})^p\psi\bar{u}|{X}u|^{p-1}+2  \bar{d}(x)(\eta\bar{\eta})^p\psi \bar{u}^p \\
   +&pa(\eta\bar{\eta})^{p-1}{X}(\eta\bar{\eta})\psi\bar{u}|{X}u|^{p-1}dx+p\bar{b}(x)(\eta\bar{\eta})^{p-1} {X}(\eta\bar{\eta})\psi \bar{u}^p. \end{aligned}
\end{equation}

Defining the functions
\begin{equation}\label{simplify}
\left\{\begin{aligned}&v=\bar{u}_k\psi^{\frac{1}{p}}=(\bar{u}_k)^{\frac{t_0}{p}+1}(\bar{u}^{(l)}_k)^{\frac{t-t_0}{p}}=(\bar{u}_k)^{q_0}(\bar{u}^{(l)}_k)^{q-q_0},\\
&w=|{X}u|\psi^{\frac{1}{p}},
\end{aligned}
\right.
\end{equation}
 then the inequality \eqref{ine1} can be simplified by
\begin{equation}\label{ine2}
\begin{aligned}
   \int_{B_{2R}-\Sigma} (t_0+\theta)|\eta\bar{\eta}w|^p dx
    \leq  & \int_{B_{2R}-\Sigma}2 |\bar{d}(x)||\eta\bar{\eta}v|^p +p|\bar{b}(x)||\eta\bar{\eta}v|^{p-1} |X(\eta\bar{\eta})v|\\
    &+|c(x)||\eta\bar{\eta}w|^{p-1}|\eta\bar{\eta}v|+pa|\eta\bar{\eta}w|^{p-1} |X(\eta\bar{\eta})v| dx.
\end{aligned}
\end{equation}
We apply Young's inequality to the last two terms of \eqref{ine2}\\
\begin{equation}\label{eq:15}\begin{aligned}
&\int_{B_{2R}-\Sigma} |c(x)||\eta\bar{\eta}w|^{p-1}|\eta\bar{\eta}v| dx\leq\int_{B_{2R}-\Sigma} \varepsilon|\eta\bar{\eta}w|^p dx + \int_{B_{2R}-\Sigma} \varepsilon^{1-p}|c(x)|^p|\eta\bar{\eta}v|^p dx.
\end{aligned}\end{equation}

\begin{equation}\label{eq:16}\begin{aligned}
&\int_{B_{2R}-\Sigma}|\eta\bar{\eta}w|^{p-1} |X(\eta\bar{\eta})v| dx
\leq \int_{B_{2R}-\Sigma} \varepsilon|\eta\bar{\eta}w|^p dx + \int_{B_{2R}-\Sigma} \varepsilon^{1-p}|X(\eta\bar{\eta})v|^p dx.
\end{aligned}\end{equation}
Choosing $\varepsilon$ small enough  we conclude that
\begin{equation}\label{eq:17}
\begin{aligned}
\int_{B_{2R}-\Sigma} |\eta\bar{\eta}w|^p dx
\leq  & C \bigg( \int_{B_{2R}-\Sigma} \bar{d}(x)|\eta\bar{\eta}v|^p
    +c(x)^p|\eta\bar{\eta}v|^p+\bar{b}(x)|\eta\bar{\eta}v|^{p-1} |X(\eta\bar{\eta})v| \\&+|X(\eta\bar{\eta})v|^p dx\bigg)
  =C(I_1+I_2+I_3+I_4).\end{aligned}
\end{equation}

For $I_1,I_2,I_3 $,
when $p<\tilde{v}$,
 applying H\"older   inequality  and Lemma \ref{sobolev} we yield that


\begin{equation}\label{eq:18}\begin{aligned}
I_1
\leq& ||\bar{d}(x)||_{L^\frac{\tilde{v}}{p-\frac{\varepsilon}{2}}(B_{2R}-\Sigma)}||\eta\bar{\eta}v||_{L^\frac{\tilde{v}p}{\tilde{v}-p+\frac{\varepsilon}{2}}(B_{2R}-\Sigma)}^p\leq C ||\bar{d}(x)||_{L^\frac{\tilde{v}}{p-\varepsilon}(B_{2R}-\Sigma)}||\eta\bar{\eta}v||_{L^\frac{\tilde{v}p}{\tilde{v}-p+\frac{\varepsilon}{2}}(B_{2R}-\Sigma)}^p
\\ \leq & C||\bar{d}(x)||_{L^\frac{\tilde{v}}{p-\varepsilon}(B_{2R}-\Sigma)}|| ||\eta\bar{\eta}v||_{L^p(B_{2R}-\Sigma)}^{\frac{\varepsilon}{2}} ||\eta\bar{\eta}v||_{L^{p^{\ast}}(B_{2R}-\Sigma)}^{p-\frac{\varepsilon}{2}}\\
\leq& C||\bar{d}(x)||_{L^\frac{\tilde{v}}{p-\varepsilon}(B_{2R}-\Sigma)}||\eta\bar{\eta}v||_{L^p(B_{2R}-\Sigma)}^{\frac{\varepsilon}{2}} (||X(\eta\bar{\eta})v||_{L^p(B_{2R}-\Sigma)}^{p-\frac{\varepsilon}{2}}+||\eta\bar{\eta}Xv||_{L^p(B_{2R}-\Sigma)}^{p-\frac{\varepsilon}{2}}),
\end{aligned}\end{equation}

\begin{equation}\label{eq:20}\begin{aligned}
I_2
\leq& ||c(x)||^p_{L^\frac{\tilde{v}}{1-\frac{\varepsilon}{2}}(B_{2R}-\Sigma)} ||\eta\bar{\eta}v||_{L^\frac{\tilde{v}p}{\tilde{v}-p+\frac{p\varepsilon}{2}}(B_{2R}-\Sigma)}^p\leq C ||c(x)||^p_{L^\frac{\tilde{v}}{1-\varepsilon}(B_{2R}-\Sigma)} ||\eta\bar{\eta}v||_{L^\frac{\tilde{v}p}{\tilde{v}-p+\frac{p\varepsilon}{2}}(B_{2R}-\Sigma)}^p
\\
\leq& C ||c(x)||^p_{L^\frac{\tilde{v}}{1-\varepsilon}(B_{2R}-\Sigma)} ||\eta\bar{\eta}v||_{L^p(B_{2R}-\Sigma)}^{\frac{p\varepsilon}{2}} ||\eta\bar{\eta}v||_{L^{p^{\ast}}(B_{2R}-\Sigma)}^{p(1-\frac{\varepsilon}{2})}\\
\leq& C||c(x)||^p_{L^\frac{\tilde{v}}{1-\varepsilon}(B_{2R}-\Sigma)} ||\eta\bar{\eta}v||_{L^p(B_{2R}-\Sigma)}^{\frac{p\varepsilon}{2}} (||{X}(\eta\bar{\eta})v||_{L^p(B_{2R}-\Sigma)}^{p(1-\frac{\varepsilon}{2})}+||\eta\bar{\eta}{X}v||_{L^p(B_{2R}-\Sigma)}^{p(1-\frac{\varepsilon}{2})}),
\end{aligned}\end{equation}

\begin{equation}\label{eq:19}\begin{aligned}
 I_3
\leq& ||\bar{b}(x)||_{L^\frac{\tilde{v}}{p-1}(B_{2R}-\Sigma)} ||{X}(\eta\bar{\eta})v||_{L^p(B_{2R}-\Sigma)} ||\eta\bar{\eta}v||_{L^{p^{\ast}}(B_{2R}-\Sigma)}^{p-1}\\
\leq & C||\bar{b}(x)||_{L^\frac{\tilde{v}}{p-1}(B_{2R}-\Sigma)} ||{X}(\eta\bar{\eta})v||_{L^p(B_{2R}-\Sigma)} (||{X}(\eta\bar{\eta})v||_{L^p(B_{2R}-\Sigma)}^{p-1}+||\eta\bar{\eta}{X}v||_{L^p(B_{2R}-\Sigma)}^{p-1} ) \\ \leq &C ||\bar{b}(x)||_{L^\frac{\tilde{v}}{p-1}(B_{2R}-\Sigma)} (||{X}(\eta\bar{\eta})v||_{L^p(B_{2R}-\Sigma)}^{p}+||{X}(\eta\bar{\eta})v||_{L^p(B_{2R}-\Sigma)}||\eta\bar{\eta}{X}v||_{L^p(B_{2R}-\Sigma)}^{p-1} )
\end{aligned}\end{equation}
with $p^{\ast}=\frac{p\tilde{v}}{\tilde{v}-p}$.
When $p=\tilde{v}$,  we have
 \begin{equation}\begin{aligned}\label{325}
I_1
\leq& ||\bar{d}(x)||_{L^{\frac{\tilde{v}}{\tilde{v}-\varepsilon}}(B_{2R}-\Sigma)}||\eta\bar{\eta}v||_{L^{\frac{\tilde{v}^2}{\varepsilon}}(B_{2R}-\Sigma)}^{\tilde{v}}
\\ \leq & ||\bar{d}(x)||_{L^{\frac{\tilde{v}}{\tilde{v}-\varepsilon}}(B_{2R}-\Sigma)}|| ||\eta\bar{\eta}v||_{L^{\tilde{v}}(B_{2R}-\Sigma)}^{\frac{\varepsilon}{2}} ||\eta\bar{\eta}v||_{L^{\frac{2\tilde{v}^2-\tilde{v}\varepsilon}{\varepsilon}}(B_{2R}-\Sigma)}^{\tilde{v}-\frac{\varepsilon}{2}}\\
\leq& C||\bar{d}(x)||_{L^{\frac{\tilde{v}}{\tilde{v}-\varepsilon}}(B_{2R}-\Sigma)}||\eta\bar{\eta}v||_{L^{\tilde{v}}(B_{2R}-\Sigma)}^{\frac{\varepsilon}{2}} (||X(\eta\bar{\eta})v||_{L^{\tilde{v}}(B_{2R}-\Sigma)}^{\tilde{v}-\frac{\varepsilon}{2}}+||\eta\bar{\eta}Xv||_{L^{\tilde{v}}(B_{2R}-\Sigma)}^{\tilde{v}-\frac{\varepsilon}{2}}),
\end{aligned}\end{equation}

\begin{equation}\begin{aligned}\label{eq:112}
I_2
\leq& ||c(x)||^{\tilde{v}}_{L^{\frac{\tilde{v}}{1-\varepsilon}}(B_{2R}-\Sigma)} ||\eta\bar{\eta}v||_{L^{\frac{\tilde{v}}{\varepsilon}}(B_{2R}-\Sigma)}^{\tilde{v}}\leq ||c(x)||^{\tilde{v}}_{L^{\frac{\tilde{v}}{1-\varepsilon}} (B_{2R}-\Sigma)} ||\eta\bar{\eta}v||_{L^{\tilde{v}}(B_{2R}-\Sigma)}^{\frac{\tilde{v}\varepsilon}{2}} ||\eta\bar{\eta}v||_{L^{\frac{(2-\varepsilon){\tilde{v}}}{\varepsilon}}(B_{2R}-\Sigma)}^{\tilde{v}(1-\frac{\varepsilon}{2})}\\
\leq& C||c(x)||^{\tilde{v}}_{L^{\frac{\tilde{v}}{1-\varepsilon}} (B_{2R}-\Sigma)} ||\eta\bar{\eta}v||_{L^{\tilde{v}}(B_{2R}-\Sigma)}^{\frac{{\tilde{v}}\varepsilon}{2}} (||{X}(\eta\bar{\eta})v||_{L^{\tilde{v}}(B_{2R}-\Sigma)}^{{\tilde{v}}(1-\frac{\varepsilon}{2})}+||\eta\bar{\eta}{X}v||_{L^{\tilde{v}}(B_{2R}-\Sigma)}^{{\tilde{v}}(1-\frac{\varepsilon}{2})}).
\end{aligned}\end{equation}
 \begin{equation}\begin{aligned}\label{326}
I_3
\leq& ||\bar{b}(x)||_{L^\frac{\tilde{v}}{\tilde{v}-1-\varepsilon}(B_{2R}-\Sigma)} ||{X}(\eta\bar{\eta})v||_{L^{\tilde{v}}(B_{2R}-\Sigma)}||\eta\bar{\eta}v||_{L^\frac{\tilde{v}(\tilde{v}-1)}{\varepsilon}(B_{2R}-\Sigma)}^{\tilde{v}-1}\\
\leq & C||\bar{b}(x)||_{L^\frac{\tilde{v}}{\tilde{v}-1-\varepsilon}(B_{2R}-\Sigma)} ||{X}(\eta\bar{\eta})v||_{L^{\tilde{v}}(B_{2R}-\Sigma)} (||{X}(\eta\bar{\eta})v||_{L^{\tilde{v}}(B_{2R}-\Sigma)}^{\tilde{v}-1}+||\eta\bar{\eta}{X}v||_{L^{\tilde{v}}(B_{2R}-\Sigma)}^{\tilde{v}-1} ) \\
 =&C||\bar{b}(x)||_{L^\frac{\tilde{v}}{\tilde{v}-1-\varepsilon}(B_{2R}-\Sigma)}(||{X}(\eta\bar{\eta})v||_{L^{\tilde{v}}(B_{2R}-\Sigma)}^{\tilde{v}} + ||{X}(\eta\bar{\eta})v||_{L^{\tilde{v}}(B_{2R}-\Sigma)} ||\eta\bar{\eta}{X}v||_{L^{\tilde{v}}(B_{2R}-\Sigma)} ^{{\tilde{v}} -1}),
\end{aligned}\end{equation}
Inserting these estimates into inequality \eqref{eq:17}, we obtain
\begin{equation}\label{ine4}
  \begin{aligned} ||\eta\bar{\eta}w||_{L^p(B_{2R}-\Sigma)}^p
    \leq & C\bigg( ||\eta\bar{\eta}v||_{L^p(B_{2R}-\Sigma)}^{\frac{\varepsilon}{2}}\left(||{X}(\eta\bar{\eta})v||_{L^p(B_{2R}-\Sigma)}^{p-{\frac{\varepsilon}{2}}}+||\eta\bar{\eta}{X}v||_{L^p(B_{2R}-\Sigma)}^{p-{\frac{\varepsilon}{2}}}\right)\\
    &+||\eta\bar{\eta}v||_{L^p(B_{2R}-\Sigma)}^{{\frac{p\varepsilon}{2}}}\left(||{X}(\eta\bar{\eta})v||_{L^p(B_{2R}-\Sigma)}^{p(1-{\frac{\varepsilon}{2}})}+||\eta\bar{\eta}{X}v||_{L^p(B_{2R}-\Sigma)}^{p(1-{\frac{\varepsilon}{2}})}\right) \\
     &+||{X}(\eta\bar{\eta})v||_{L^p(B_{2R}-\Sigma)}^p +  ||{X}(\eta\bar{\eta})v||_{L^p(B_{2R}-\Sigma)}||\eta\bar{\eta}{X}v||_{L^p(B_{2R}-\Sigma)}^{p-1} \bigg),
    \end{aligned}
\end{equation}
By \eqref{eq:13}and  \eqref{eq:14} we note that
$$\begin{aligned}
|{X}v|&\leq |\psi^{\frac{1}{p}}{X}u|+\frac{1}{p}|\psi^{\frac{1}{p}}{X}u|\{t \chi _{k,l}+t_0 \chi_l\}\leq (1+\frac{t}{p}) |\psi^{\frac{1}{p}}{X}u|=q|\psi^{\frac{1}{p}}{X}u|=qw.
\end{aligned}$$
Since $q\leq 1$, then
 \begin{equation}\label{eq:21}
 ||\eta\bar{\eta}Xv||_{L^p(B_{2R}-\Sigma)}^p \leq ||\eta\bar{\eta}w||_{L^p(B_{2R}-\Sigma)}^p\end{equation}
We  set
$$z=\frac{||\eta\bar{\eta}w||_{L^p(B_{2R}-\Sigma)}}{||X(\eta\bar{\eta})v||_{L^p(B_{2R}-\Sigma)}},\quad \xi=\frac{||\eta\bar{\eta} v||_{L^p(B_{2R}-\Sigma)}}{||X(\eta\bar{\eta})v||_{L^p(B_{2R}-\Sigma)}},$$
by \eqref{ine4} and \eqref{eq:21}, we obtain
\begin{equation}\label{ine5}
z^p \leq C\{\xi^{\frac{\varepsilon}{2}}(1+z^{p-\frac{\varepsilon}{2}})+\xi^{\frac{p\varepsilon}{2}}(1+z^{p(1-\frac{\varepsilon}{2})})+(1+z^{p-1})+1\}.
\end{equation}
By  Lemma \ref{inez}, we can get the estimate $z\leq C (1+\xi)$, that is,
\begin{equation}\begin{aligned}\label{first1}
&||\eta\bar{\eta}Xv||_{L^p(B_{2R}-\Sigma)} \leq ||\eta\bar{\eta}w||_{L^p(B_{2R}-\Sigma)} \\
\leq& C(||\bar{\eta}(X\eta) v||_{L^p(B_{2R}-\Sigma)} +||\eta (X\bar{\eta} ) v||_{L^p(B_{2R}-\Sigma)}+||\eta\bar{\eta}v||_{L^p(B_{2R}-\Sigma)}).\end{aligned}
\end{equation}
Meanwhile, by Lemma \ref{sobolev} we yield that
\begin{equation}\label{first2}
||\eta\bar{\eta}v||_{L^{p^{\ast}}(B_{2R}-\Sigma)} \leq C(||\bar{\eta}(X\eta) v||_{L^p(B_{2R}-\Sigma)} +||\eta (X\bar{\eta}) v||_{L^p(B_{2R}-\Sigma)}+||\eta\bar{\eta}v||_{L^p(B_{2R}-\Sigma)}).
\end{equation}
Here when $ p<\tilde{v},p^{*}=\frac{\tilde{v}p}{\tilde{v}-p}$; when $p=\tilde{v},p^{*}=2p$.

Now we take $\bar{\eta}=\bar{\eta}_v$, where $\bar{\eta}_v$ is constructed in Lemma \ref{caps}. By H\"older inequality,  we have
$$\begin{aligned}||\eta (X\bar{\eta}_v) v||_{L^p(B_{2R}-\Sigma)} &\leq ||{X}\bar{\eta}_v||_{L^s(B_{2R}-\Sigma)} ||v||_{L^\frac{sp}{s-p}(B_{2R}-\Sigma)}
\leq ||{X}\bar{\eta}_v||_{L^s(B_{2R}-\Sigma)} ||l^{q-q_0}{\bar{u}_k}^{q_0} ||_{L^\frac{sp}{s-p}(B_{2R}-\Sigma)}.\end{aligned} $$
By  \eqref{eq:133} and  Lemma \ref{caps},
as $ v \rightarrow \infty,$ $$||\eta (X\bar{\eta}_v) v||_{L^p(B_{2R}-\Sigma)} \leq ||X\bar{\eta}_v||_{L^s(B_{2R}-\Sigma)} l^{q-q_0} ||{\bar{u}_k}||_{L^{\frac{s(p-\theta)}{s-p}(1+\delta)}(B_{2R}-\Sigma)}\to 0.$$
Letting $v \rightarrow \infty$  and applying the dominated convergence theorem, \eqref{first1} and  \eqref{first2} can be simplified to
\begin{equation}\label{first12}
||\eta w||_{L^p(B_{2R}-\Sigma)} \leq C(|| {X}\eta v||_{L^p(B_{2R}-\Sigma)} +||\eta v||_{L^p(B_{2R}-\Sigma)}).\end{equation}
\begin{equation}\label{eq:22}
||\eta v||_{L^{p^{\ast}}(B_{2R}-\Sigma)} \leq C(||{X}\eta v||_{L^p(B_{2R}-\Sigma)} +||\eta v||_{L^p(B_{2R}-\Sigma)}).\end{equation}
 Here when $p<\tilde{v}$, $p^{*}=\frac{\tilde{v}p}{\tilde{v}-p}$ ; when  $ p=\tilde{v} $ $p^{*}=2p$.
Note also that as $l \rightarrow \infty$,
$$\left\{\begin{aligned}&\psi =( \bar{u}_k)^{t_0}(\bar{u}^{(l)}_k)^{t-t_0} \rightarrow ( \bar{u}_k)^{t}  ,\\
&v=\bar{u}_k\psi^{\frac{1}{p}}=(\bar{u}_k)^{q_0}(\bar{u}^{(l)}_k)^{q-q_0} \rightarrow ( \bar{u}_k)^{q}  ,\\
&w=|{X}u|\psi^{\frac{1}{p}} \rightarrow  |{X}u|( \bar{u}_k)^{q-1}.
\end{aligned}
\right. $$
  By \eqref{first12}, \eqref{eq:22} and monotone convergence theorem, we yield that
\begin{equation}\label{eq:23}||\eta  (\bar{u}_k)^{q-1}{X}u ||_{L^p(B_{2R}-\Sigma)} \leq C(|| ( \bar{u}_k)^{q}{X}\eta||_{L^p(B_{2R}-\Sigma)} +||\eta ( \bar{u}_k)^{q}||_{L^p(B_{2R}-\Sigma)}).\end{equation}
\begin{equation}\label{first13}
\begin{aligned}
||\eta \bar{u}_k^q||_{L^{p^{\ast}}(B_{2R}-\Sigma)} \leq C(|| ( \bar{u}_k)^{q}{X}\eta||_{L^p(B_{2R}-\Sigma)}+||\eta ( \bar{u}_k)^{q}||_{L^p(B_{2R}-\Sigma)}).
\end{aligned}
\end{equation}

Now  let $r_0 \in \mathbb{N}$ such that
$$(\frac{p}{p*})^{r_0+1} <q_0\leq (\frac{p}{p*})^{r_0}, $$
and  take
$q=q_j=(\frac{p*}{p})^{j}q_0$
where $j=0,1...,r_0$. Then we
note the fact that $$p\cdot q_{j}=p^{\ast}\cdot q_{j-1}.$$ Selecting appropriate functions $\eta$, after finite iteration of \eqref{first13}  we obtain that
\begin{equation}\label{itera}\begin{aligned}
||\bar{u}_k||_{L^p(B_{\frac{3R}{2}}-\Sigma)} \leq &C || \bar{u}_k||_{L^{p^{\ast}\cdot {q_{r_0}}}(B_{\frac{3R}{2}}-\Sigma)}
\leq  C|| \bar{u}_k||_{L^{p\cdot q_{r_0}}(B_{\frac{7R}{4}}-\Sigma)}=C|| \bar{u}_k||_{L^ { q_{(r_0-1)}\cdot p^{\ast}}(B_{\frac{7R}{4}}-\Sigma)}
\\ \leq&  C|| \bar{u}_k||_{L^{p\cdot {q_{0}}}(B_{2R}-\Sigma)}\leq  C|| \bar{u}_k||_{L^{\frac{s(p-\theta)}{s-p}(1+\delta)}(B_{2R}-\Sigma)}
<\infty.\end{aligned}
\end{equation}

Finally, taking $q=1$ in inequality \eqref{eq:23} and \eqref{first13} we have
\begin{equation}\label{first14}
\begin{aligned}
&||u||_{L^p(B_R-\Sigma)} \leq ||\bar{u}_k||_{L^p(B_R-\Sigma)}  < \infty.\\
&||{X}u ||_{L^p(B_R-\Sigma)} \leq C ||\bar{u}_k||_{L^p(B_R-\Sigma)}< \infty.
\end{aligned}
\end{equation}

Step2: we will show that $u\in \mathnormal{W}^{1,p}_{X}(B_R)$ and is a weak solution of \eqref{equa} in $B_R$.

Firstly, we set $u=0, {X}u=0$ on $\Sigma$ and construct the function sequence
\begin{equation}\label{eq:25}
\widetilde{u}_m =\min\{\max (u,-m),m\},  \ \ \  m>0. \end{equation}
Obviously, $\widetilde{u}_m \in \mathnormal{W}^{1,p}_{X}\cap L^{\infty}(B_R-\Sigma).$   According to the Definition \ref{def2.1}, for any function $\varphi \in \mathnormal{C}^{\infty}_0(B_R)$, we have that
$$\int_{B_R}( {X}_i \widetilde{u}_m) \varphi\bar{\eta}_v dx=\int_{B_R} \widetilde{u}_m \bar{\eta}_v {X}^{*}_i \varphi dx-\int_{B_R} \widetilde{u}_m  \varphi  {X}_i\bar{\eta}_v dx, $$
where $X^{*}_i \varphi =-\sum\limits_{j=1}^{n}\partial_{x_j}(h_{ij}(x)\varphi),$  and  $\bar{\eta}_v$ is the function in Lemma \ref{caps}. As $v \rightarrow \infty$, by the boundness of $\widetilde{u}_m$, we obtain that
 $$\begin{aligned}
 \left|\int_{B_R} \widetilde{u}_m  \varphi  {X}\bar{\eta}_v dx\right| \leq ||{X}\bar{\eta}_v||_{L^{s}(B_R)} ||\widetilde{u}_m  \varphi ||_{L^\frac{s}{s-1}(B_R)}\to 0.\end{aligned}$$
 Then by  the dominated convergence theorem, the following is holds
$$\int_{B_R} {X}_{i} \widetilde{u}_m \varphi dx=\int_{B_R} \widetilde{u}_m  {X}^{*}_{i} \varphi dx.$$
Therefore, $\widetilde{u}_m \in \mathnormal{W}^{1,p}_{X}\cap L^{\infty}(B_R).$ In addition, due to $|\widetilde{u}_m|\leq|u|,|{X} \widetilde{u}_m|\leq|{X}u|$ on $B_R$,  $||\widetilde{u}_m||_{\mathnormal{W}^{1,p}_{X}(B_R)}$ is uniformly bounded. Hence there exists  a subsequence of $\widetilde{u}_m$ converges weakly in $\mathnormal{W}^{1,p}_{X}(B_R)$. On the other hand, $\widetilde{u}_m$ converges strongly in $L^p(B_R)$ to $u$ by \eqref{eq:25}. Thus we obtain that $u\in \mathnormal{W}^{1,p}_{X}(B_R)$.

Secondly, if $\bar{\eta}_v$ is the function in Lemma \ref{caps},  for any $\phi \in \mathnormal{C}^{\infty}_0(B_R)$ we have
\begin{equation}\label{solu}
\begin{aligned}
0=&\int_{B_R} [{X}(\phi\bar{\eta}_v)\cdot A(x,u,{X}u)-\phi\bar{\eta}_v B(x,u,{X}u) ]dx\\
=&\int_{B_R}{\bar{\eta}}_v  ({X}\phi\cdot A(x,u,{X}u)-\phi B(x,u,{X}u)) dx+\int_{B_R} \phi A(x,u,{X}u)\cdot {X}\bar{\eta}_v dx.
\end{aligned}
\end{equation}
Since $u\in \mathnormal{W}^{1,p}_{X}(B_R)$, by \eqref{cond} we have
$ A(x,u,Xu)\in L^{\frac{p}{p-1}}(B_R)$.
As $v \rightarrow \infty$,
$$\begin{aligned}\left|\int_{B_R} \phi A(x,u,{X}u)\cdot {X}\bar{\eta}_v dx\right|\leq& ||{X}\bar{\eta}_v||_{L^{s}(B_R)} ||\phi A(x,u, {X}u) ||_{L^{\frac{s}{s-1}}(B_R)}\\
\leq & C ||{X}\bar{\eta}_v||_{L^{s}(B_R)} || A(x,u,{X}u) ||_{L^{\frac{p}{p-1}}(B_R)}\to 0. \end{aligned}$$
Thus \eqref{solu} yields
$$\int_{B_R} X\phi\cdot A-\phi B dx=0,$$
which implies  $u$ is a weak solution of \eqref{equa} in $B_R$.
\end{proof}

\begin{proof}[\textbf{Proof of Theorem \ref{Thmholderctn}}] Analogous to Theorem \ref{Thmremovsing}, we consider that for any open ball  $B_{2R}\subset\subset \Omega $, the function $ u\in \mathnormal{W}^{1,p}_{X}(B_R)$ and  $u$ is a  bounded weak solution of \eqref{equa} in $B_R.$ The key is to prove that $||Xu||_{L^p(B_R-\Sigma)}$ is finite.

Let $$ k=1+(||e||_{B_{2R}}+||f||_{B_{2R}})^{\frac{1}{p-1}}+||g||_{B_{2R}}^{\frac{1}{p}},$$ where $||e||,||f||$ and $||g||$ are as defined in the corresponding norm of \eqref{func}. Setting $$\bar{u}=|u|+k,\  \bar{b}(x)=b(x)+k^{1-p}e(x), \ \bar{d}(x)=d(x)+k^{1-p}f(x)+k^{-p}g(x),$$ we note that \eqref{thmesti1} still holds.

Let $\phi=(\eta\bar{\eta}_v)^p u e^{c_0 |u|},$ where $\eta\in \mathnormal{C}^{\infty}_0(B_{2R})$  with $\eta =1$ on $B_R,$ $0\leq \eta \leq 1,$  and $\bar{\eta}_v$ is the function in Lemma \ref{caps}. By lemma \ref{chain3} we get that $\phi \in \mathnormal{W}^{1,p}_{X,0}(B_{2R}-\Sigma)$, and  we have
$$ X \phi=p(\eta\bar{\eta}_v)^{p-1}X(\eta\bar{\eta}_v)u e^{c_0 |u|}+(1+c_0|u|)(\eta\bar{\eta}_v)^p e^{c_0 |u|} Xu.$$

Based on the fact that $u$ is a weak solution of \eqref{equa} in $B_{2R}-\Sigma$, choosing  $\phi$ as the test function, we have
\begin{equation}\label{eq:24}
 \begin{aligned}&\int_{B_{2R}-\Sigma}  (\eta\bar{\eta}_v)^p |X \bar{u}|^p dx
\leq C\bigg( \int_{B_{2R}-\Sigma} \bar{d}(x)|\eta\bar{\eta}_v|^p dx+\int_{B_{2R}-\Sigma} c(x)|\eta\bar{\eta}_v|^p|X \bar{u}|^{p-1} dx\\+&\int_{B_{2R}-\Sigma} |\eta\bar{\eta}_v|^{p-1}|X(\eta\bar{\eta}_v)||X \bar{u}|^{p-1} dx+\int_{B_{2R}-\Sigma} \bar{b}(x) |\eta\bar{\eta}_v|^{p-1}|X(\eta\bar{\eta}_v)| dx\bigg),
\end{aligned}
 \end{equation}
where $C$ depends on $n,p,a,k,c_0,||u||_{L^{\infty}(\Omega)}.$

By   Young inequality, we have
$$ \int_{B_{2R}-\Sigma} c(x)|\eta\bar{\eta}_v|^p|X \bar{u}|^{p-1} dx\leq\int_{B_{2R}-\Sigma} \varepsilon |\eta\bar{\eta}_v|^p |X \bar{u}|^p dx + \int_{B_{2R}-\Sigma}  \varepsilon^{1-p}c(x)^p |\eta\bar{\eta}_v|^p dx,$$
$$\int_{B_{2R}-\Sigma} |\eta\bar{\eta}_v|^{p-1}|X(\eta\bar{\eta}_v)||X \bar{u}|^{p-1} dx\leq \int_{B_{2R}-\Sigma} \varepsilon |\eta\bar{\eta}_v|^p |X \bar{u}|^p dx + \int_{B_2-\Sigma} \varepsilon^{1-p} |X(\eta\bar{\eta}_v)|^p dx.$$
Then \eqref{eq:24} can be simplified to
\begin{equation}\label{511}
 \begin{aligned}&\int_{B_{2R}-\Sigma}  (\eta\bar{\eta}_v)^p |X \bar{u}|^p dx
\leq C\bigg( \int_{B_{2R}-\Sigma} \bar{d}(x)|\eta\bar{\eta}_v|^p dx+\int_{B_{2R}-\Sigma} c(x)^p |\eta\bar{\eta}_v|^p dx\\+&\int_{B_{2R}-\Sigma} |X(\eta\bar{\eta}_v)|^p dx+\int_{B_{2R}-\Sigma} \bar{b}(x) |\eta\bar{\eta}_v|^{p-1}|X(\eta\bar{\eta}_v)| dx\bigg)=C(I_1+I_2+I_3+I_4).
\end{aligned}
 \end{equation}

For $I_1,I_2,I_3$, when $p \leq \tilde{v},$ applying H\"older inequality and Lemma \ref{sobolev} we obtain
\begin{equation}\label{eq:52}\begin{aligned}
I_1\leq& ||\bar{d}(x)||_{L^{\frac{\tilde{v}}{p-\varepsilon}}(B_{2R}-\Sigma)}||\eta\bar{\eta}_v||_{L^{\frac{\tilde{v}p}{\tilde{v}-p+\varepsilon}}(B_{2R}-\Sigma)}^p
\leq C||\bar{d}(x)||_{L^{\frac{\tilde{v}}{p-\varepsilon}}(B_{2R}-\Sigma)}|| ||\eta\bar{\eta}_v||_{L^{p^{\ast}}(B_{2R}-\Sigma)}^p\\
\leq& C||\bar{d}(x)||_{L^{\frac{\tilde{v}}{p-\varepsilon}}(B_{2R}-\Sigma)}\left(||\bar{\eta}_vX\eta||_{L^p(B_{2R}-\Sigma)}^p + ||\eta X\bar{\eta}_v||_{L^p(B_{2R}-\Sigma)}^p\right),\\
\end{aligned}\end{equation}
here  $p^{\ast}=\frac{p\tilde{v}}{\tilde{v}-p},$ when $ p<\tilde{v};\ p^{\ast}=\frac{{\tilde{v}}^2}{\varepsilon}$, when $ p=\tilde{v}$.
\begin{equation}\label{eq:53}\begin{aligned}
I_2
\leq& ||c(x)||_{L^{\frac{\tilde{v}}{1-\varepsilon}}(B_{2R}-\Sigma)}^p||\eta\bar{\eta}_v||_{L^{\frac{\tilde{v}p}{\tilde{v}-p+p\varepsilon}}(B_{2R}-\Sigma)}^p
\leq C||c(x)||_{L^{\frac{\tilde{v}}{1-\varepsilon}}(B_{2R}-\Sigma)}^p||\eta\bar{\eta}_v||_{L^{p^{\ast}}(B_{2R}-\Sigma)}^p\\
\leq& C||c(x)||_{L^{\frac{\tilde{v}}{1-\varepsilon}}(B_{2R}-\Sigma)}^p\left(||\bar{\eta}_v X\eta||_{L^p(B_{2R}-\Sigma)}^p + ||\eta X\bar{\eta}_v||_{L^p(B_{2R}-\Sigma)}^p\right),\\
\end{aligned}\end{equation}
here  $p^{\ast}=\frac{p\tilde{v}}{\tilde{v}-p},$ when $ p<\tilde{v};\ p^{\ast}=\frac{{\tilde{v}}}{\varepsilon}$, when $ p=\tilde{v}$.
\begin{equation}\label{eq:54}\begin{aligned}
I_3
\leq& \int_{B_{2R}-\Sigma} ||X(\eta\bar{\eta}_v)||_{L^p(B_{2R}-\Sigma)}^p dx
 =C\left(||\bar{\eta}_v X\eta||_{L^p(B_{2R}-\Sigma)}^p + ||\eta X\bar{\eta}_v||_{L^p(B_{2R}-\Sigma)}^p\right).\\
\end{aligned}\end{equation}

For $I_4$,
when $p<\tilde{v},$  we can get the following estimates
\begin{equation}\label{eq:55}\begin{aligned}
I_4
\leq& ||\bar{b}(x)||_{L^{\frac{\tilde{v}}{p-1}}(B_{2R}-\Sigma)}||X(\eta\bar{\eta}_v)||_{L^p(B_{2R}-\Sigma)}||\eta\bar{\eta}_v||_{L^{p^{\ast}}(B_{2R}-\Sigma)}^{p-1}\\
\leq& C||\bar{b}(x)||_{L^{\frac{\tilde{v}}{p-1}}(B_{2R}-\Sigma)}||X(\eta\bar{\eta}_v)||_{L^p(B_{2R}-\Sigma)}^p
\\ \leq & C||\bar{b}(x)||_{L^{\frac{\tilde{v}}{p-1}}(B_{2R}-\Sigma)}\left(||\bar{\eta}_v X\eta||_{L^p(B_{2R}-\Sigma)}^p + ||\eta X\bar{\eta}_v||_{L^p(B_{2R}-\Sigma)}^p\right),\\
\end{aligned}\end{equation}
where $p^{\ast}=\frac{p\tilde{v}}{\tilde{v}-p};$ when $p=\tilde{v}$ ,we get that
\begin{equation}\label{eq:5555}\begin{aligned}
I_4
\leq& ||\bar{b}(x)||_{L^{\frac{\tilde{v}}{{\tilde{v}}-1-\varepsilon}}(B_{2R}-\Sigma)}||X(\eta\bar{\eta}_v)||_{L^{\tilde{v}}(B_{2R}-\Sigma)}||\eta\bar{\eta}_v||_{L^{\frac{{\tilde{v}}({\tilde{v}}-1)}{\varepsilon}}(B_{2R}-\Sigma)}^{{\tilde{v}}-1}\\
\leq& C||\bar{b}(x)||_{L^{\frac{\tilde{v}}{{\tilde{v}}-1-\varepsilon}}(B_{2R}-\Sigma)}||X(\eta\bar{\eta}_v)||_{L^{\tilde{v}}(B_{2R}-\Sigma)}^{\tilde{v}}\\
\leq& C||\bar{b}(x)||_{L^{\frac{\tilde{v}}{{\tilde{v}}-1-\varepsilon}}(B_{2R}-\Sigma)}\left(||\bar{\eta}_v X\eta||_{L^{\tilde{v}}(B_{2R}-\Sigma)}^{\tilde{v}} + ||\eta X\bar{\eta}_v||_{L^{\tilde{v}}(B_{2R}-\Sigma)}^{\tilde{v}}\right).\\
\end{aligned}\end{equation}

Inserting these estimates into inequality \eqref{511}, we obtain that
\begin{equation}\label{561}
\int_{B_{2R}-\Sigma}   (\eta\bar{\eta}_v)^p |X \bar{u}|^p dx
\leq C (||\bar{\eta}_vX\eta||_{L^p(B_{2R}-\Sigma)}^p + ||\eta X\bar{\eta}_v||_{L^s(B_{2R}-\Sigma)}^p).
\end{equation}
 Letting $v \rightarrow \infty,$ by Lemma \ref{caps}, we get
\begin{equation}\label{562}
\int_{B_R-\Sigma} |X \bar{u}|^p dx  \leq C,
\end{equation}
where $C$ depends on $n,p,\varepsilon ,a,c_0,k,||d||_{\Omega},||c||_{\Omega},||b||_{\Omega},||u||_{L^{\infty}(\Omega)}.$ That is $||Xu||_{L^p(B_R-\Sigma)}$ is finite, then the rest follows the process of  proving Theorem \ref{Thmremovsing} and we can get the conclusion.

\end{proof}

\section{ Harnack inequality and  H\"older continuity }
In this section, we give the proof of Theorem \ref{Thm5.1} and Theorem \ref{Thmholderctn1} to  obtain the Harnack inequality and H\"older continuity  for non-negative bounded weak solutions of \eqref{equa}.  Before that, let's prepare by giving some of the existing lemmas.
\begin{lemma}[\cite{DN}, Lemma 8.23]\label{lemma3.7}
Let $\omega$ be a non-decreasing function on an interval $(0, r_0]$ satisfying  the inequality
$$\omega(\tau r)\leq \gamma \omega(r)+\sigma(r), \ \forall \ r\leq r_0,$$
where $\sigma$ is also non-decreasing and $0<\gamma, \tau<1$. Then for any $\mu\in(0,1)$ and $r\leq r_0$, we have
$$ \omega(r)\leq C(\frac{r}{r_0})^{\beta}\omega (r_0)+\sigma (r^{\mu }r_0^{1-\mu}),$$
where $C=C(\gamma,\tau)$ and $\beta=\beta(\gamma,\tau,\mu)$ are positive constants.
\end{lemma}
\begin{lemma}[\cite{GSP}, Proposition 2.4]\label{doubling}
For any compact subset $K\subset U$, there exist $C >1$ and $\rho_K > 0$ such that
$$|B(x,2r)|\leq C|B(x,r)|  \ \ \ for \ all \ x\in K,\  0<r\leq \frac{\rho_K}{2}. $$
\end{lemma}
\begin{lemma}[\cite{VMO}, Lemma 8.23]\label{cutoff}
Let $\Omega$ is a  bounded open  domain of $U$, and  $B(x_0,r_1)\subset B(x_0,r_2)\subset\Omega$.  Then  there exists a function  $\eta\in C(B(x_0,r_2))\cap W_{X,0}^{1,q}(B(x_0,r_2))$,  $q\in  [1,+\infty)$,  which  satisfies $supp \eta\in B(x_0,r_2),$ $ 0\leq \eta\leq 1,$ $\eta \equiv 1 $  on $B(x_0,r_1)$  and
 $$ |X\eta|\leq \frac{C}{(r_2-r_1)}, \ a.e. \  in   \ B(x_0,r_2) . $$
\end{lemma}
\begin{lemma}%
[(poincar\'e)\cite{J},  Theorem 1]\label{poincare}
Let $W$ be an connected open set in $\mathbb{R}^{n}$, and   $X_1,...,X_m$ be a collection of $\mathnormal{C}^{\infty}$ vector fields defined in a neighborhood $W_0$ of the closure $\overline{W}$ in $\mathbb{R}^{n}$. Let $K$ be a compact subset of $W$. Then there is constant $C$ and a radius $r_0 >0$ such that for every $x \in K $ and $0<r<r_0,$ and if $1\leq p<n\gamma$ and $\frac{1}{q}=\frac{1}{p}-\frac{1}{n\gamma}$, then
$$ \left(\frac{1}{|B(x,r)|}\int_{B(x,r)} |v-v_B|^q dx \right)^{\frac{1}{q}}\leq Cr \left(\frac{1}{|B(x,r)|}\int_{B(x,r)} |Xv|^p dx \right)^{\frac{1}{p}},$$
 for any $ v\in Lip(\overline{B(x,r)})$. Here $v_B=\frac{1}{|B(x,r)|}\int_{B(x,r)} |v|dx,$ and $\gamma$ is a constant such that for any balls $I\subset J\subset B(x,r)$, $x\in K, r<r_0$, the following inequality holds, $$|J|\leq C\left(\frac{r(J)}{r(J)}\right)^{n\gamma}|I|.$$
\end{lemma}
\begin{Rem}\label{remark5.1}
By simple computation, we get
\begin{equation}\label{eq:2}\int_{B(x,r)} |v-v_B| dx \leq Cr \int_{B(x,r)} |Xv| dx \end{equation}
for any $v\in C^{\infty}(\overline{B(x,r)})$.
By approximation, \eqref{eq:2} also holds for any $v \in \mathnormal{W}^{1,1}_{X}(B(x,2r)).$
\end{Rem}
\begin{lemma}[(John-Nirenberg)\cite{JN}, Theorem 1 ]\label{jn} Let X be a sapce of homogeneous type and $f \in \mathcal{L}_{\varphi,1}(X)$ satisfying $||f||_{\mathcal{L}_{\varphi,1}(X)} \neq 0.$ Then  there exist positive constants $C_1$ and $C_2,$ which are independent of $f,$ such that for all balls $B\subset X$ and $\alpha >0,$ if $ \varphi \in \mathbb{A}_1 (X),$  then
$$\varphi\left(\{{x\in B:{\frac{|f(x)-f_B|}{\varphi(x,||\chi_B||_{L^{\varphi}}^{-1})}}>\alpha \}},||\chi_B||_{L^{\varphi}}^{-1}\right)\leq C_1 \exp \left(-\frac{C_2 \alpha}{||\chi_B||_{L^{\varphi}}||f||_{\mathcal{L}_{\varphi,1}(X)}}\right) .$$
\end{lemma}
\begin{Rem}
Particularly, by computing directly, we can verify that when $\varphi(x,t)=t$ for all $x\in X$ and $t\in [0,\infty), $ $ ||\chi_B||_{L^{\varphi}}=\mu (B).$ Thus, the above conclusion in this case is the classical John-Nirenberg inequality for sapces of homogeneous type
$$\mu\left(\{x\in B:|f(x)-f_B|>\alpha\}\right)\leq C_1 e^{\frac{-C_2\alpha}{||f||_{BMO(X)}}} \mu(B).$$
\end{Rem}

\begin{proof}[\textbf{proof of Theorem \ref{Thm5.1}}]
Let  $\bar u =u+\tilde k(R)$, $\beta \neq0$  and $$\tilde\phi=\eta^p \bar{u}^{\beta} e^{(sign\beta)c_0 \bar{u}},$$
where  $\eta\in C(B_{4R})\cap W_{X,0}^{1,q}(B_{4R})$ for any $q\in[1,\infty)$,   $0\leq\eta\leq1,$  $supp \eta\in (B_{4R}),$  and $|X \eta| \in L^{\infty}(B_{4R})$.
And we have
$$ X \tilde\phi=e^{(sign\beta)c_0 \bar{u}}(p\eta^{p-1} \bar{u}^{\beta} X\eta+ \beta \eta^p \bar{u}^{\beta-1}X{u}+(sign\beta)c_0\eta^p \bar{u}^{\beta}X{u}).$$
Note that $\tilde\phi \in\mathnormal{W}^{1,p}_{X,0}(B_{4R}),$ by approximation, \eqref{weso} also holds for any test function ${\tilde\phi}.$ Then we get
\begin{equation}\int_{B_{4R}} X {\tilde\phi} \cdot A(x,u,Xu) dx = \int_{B_{4R}} B(x,u,Xu){\tilde\phi} dx.\end{equation}
Thus we get that

$$\begin{aligned}&\int_{B_{4R}} \beta \eta^p \bar{u}^{\beta-1}X{u}+(sign\beta)c_0\eta^p \bar{u}^{\beta}X{u}\cdot A(x,u,Xu) dx
\\ \leq& \int_{B_{4R}} (\eta^p \bar{u}^{\beta}B(x,u,Xu)-p\eta^{p-1} \bar{u}^{\beta} X\eta \cdot A(x,u,Xu)) dx.\end{aligned}$$

Setting $\bar b(x)= b(x)+{\tilde k(R)}^{1-p}e(x)$ and $\bar d(x)=d(x)+{\tilde k(R)}^{1-p}f(x) +{\tilde k(R)}^{-p}g(x)$, we have
\begin{equation}\label{bd}
||\bar b(x)||_{B_{4R}}\leq ||b(x)||_{B_{4R}} +1,\;\;\;||\bar d(x)||_{B_{4R}}\leq ||d(x)||_{B_{4R}} +2|B_R|^{-\frac{\varepsilon}{2\tilde v}}.
\end{equation}
Combining \eqref{cond}, \eqref{cond_prime} and \eqref{thmesti1}, we obtain that
\begin{equation}\label{572}
 \begin{aligned}\int_{B_{4R}}   |\beta|\eta^p{\bar{u}}^{\beta-1}|X{u}|^p dx
\leq &\int_{B_{4R}} \bigg((1+|\beta|+c_0 \bar{u})\eta^p\bar{d}(x){\bar{u}}^{p+\beta-1}+ c(x)\eta^p{\bar{u}}^{\beta}|X{u}|^{p-1}\\
 & +pa\eta^{p-1}{\bar{u}}^{\beta} |X\eta||X{u}|^{p-1} dx+ p\eta^{p-1}\bar{b}(x){\bar{u}}^{p+\beta-1}|X\eta|\bigg).\\
\end{aligned}
\end{equation}
 Setting
\begin{equation}\label{573}
v(x) =\begin{cases}
\log \bar{u}(x)  \ \ \ & if \ \ \beta=1-p,\\
{\bar{u}(x)}^q   \ &if \ \beta \neq 1-p \ \  and \ \  pq=p+\beta-1,
\end{cases}
 \end{equation}
and  we  next analyze  in two cases.

Case\ 1: firstly,we consider the case $\beta=1-p .$ We can rewrite \eqref{572} as
\begin{equation}\label{581}
 \begin{aligned}\int_{B_{4R}}  (p-1)|\eta Xv|^p dx
\leq &\int_{B_{4R}}  \bigg((p+c_0 \bar{u})\bar{d}(x)\eta^p+ c(x)\eta|\eta Xv|^{p-1}\\
 & +pa|X\eta||\eta Xv|^{p-1}+ p\eta^{p-1}\bar{b}(x)|X\eta|\bigg) dx=I_1+I_2+I_3+I_4.\\
\end{aligned}
\end{equation}
For any $h \leq 2R,$ by Lemma \ref{cutoff}, we specialize $\eta$ by
$$\eta(x) =\begin{cases}
1  \  \ & if  \ \ x \in B_h,\\
0   \ \ &if   \ \ x \notin B_{2h},
\end{cases}$$
and  $|X\eta| \leq \frac{C}{h}$ a.e. in $B_{2h}$.

For $I_1,I_2,I_3$, when $p \leq \tilde{v},$ by \eqref{bd}, H\"older inequality and Lemma \ref{sobolev}  we have
\begin{equation}\label{eq:522}\begin{aligned}
I_1
\leq& C||\bar{d}(x)||_{L^{\frac{\tilde{v}}{p-\frac{\varepsilon}{2}}}(B_{2h})}||\eta||_{L^{\frac{\tilde{v}p}{\tilde{v}-p+\frac{\varepsilon}{2}}}(B_{2h})}^p
\leq C|B_{2h}|^{\frac{\varepsilon}{2p-\varepsilon }}||\bar{d}(x)||_{L^{\frac{\tilde{v}}{p-\varepsilon}}(B_{2h})} ||\eta||_{L^{p^{\ast}}(B_{2h})}^p\\
\leq& C|B_{2h}|^{\frac{\varepsilon}{2p-\varepsilon }}||\bar{d}(x)||_{L^{\frac{\tilde{v}}{p-\varepsilon}}(B_{2h})}||X\eta||_{L^p(B_{2h})}^p
\leq C h^{-p}|B_{2h}| |B_{2h}|^{\frac{\varepsilon}{2p-\varepsilon }}||\bar{d}(x)||_{L^{\frac{\tilde{v}}{p-\varepsilon}}(B_{2h})}\\
\leq& C h^{-p}|B_{2h}|\left(||d(x)||_{\frac{\tilde{v}}{p-\varepsilon}(B_{2h})}+1\right),\\
\end{aligned}\end{equation}
here when $ p<\tilde{v},p^{*}=\frac{\tilde{v}p}{\tilde{v}-p}$; when $p=\tilde{v},p^{*}=\frac{2\tilde{v}^2}{\varepsilon}$.
\begin{equation}\label{eq:532}\begin{aligned}
 I_2 \leq &C||c(x)||_{L^{\frac{\tilde{v}}{1-\varepsilon}}(B_{2h})} ||\eta||_{L^{\frac{\tilde{v}p}{\tilde{v}-p+p\varepsilon}}(B_{2h})} ||\eta Xv||_{L^p(B_{2h})}^{p-1}
\leq C||c(x)||_{L^{\frac{\tilde{v}}{1-\varepsilon}}(B_{2h})} ||\eta||_{L^{p^{\ast}}(B_{2h})} ||\eta Xv||_{L^p(B_{2h})}^{p-1}\\
\leq& C ||c(x)||_{L^{\frac{\tilde{v}}{1-\varepsilon}}(B_{2h})} ||X \eta||_{L^p(B_{2h})} ||\eta Xv||_{L^p(B_{2h})}^{p-1}
\leq C h^{-1}|B_{2h}|^{\frac{1}{p}}||c(x)||_{L^{\frac{\tilde{v}}{1-\varepsilon}}(B_{2h})}||\eta Xv||_{L^p(B_{2h})}^{p-1},\\
\end{aligned}\end{equation}
here when $ p<\tilde{v},p^{*}=\frac{\tilde{v}p}{\tilde{v}-p}$; when $p=\tilde{v},p^{*}=\frac{\tilde{v}}{\varepsilon}$.
\begin{equation}\label{eq:552}
 I_3 \leq ||\eta Xv||_{L^p(B_{2h})}^{p-1} ||X\eta||_{L^p(B_{2h})}  \leq C h^{-1}|B_{2h}|^{\frac{1}{p}}||\eta Xv||_{L^p(B_{2h})}^{p-1}.
\end{equation}

In particular, based on the different integrability of the function $b(x)$ in \eqref{func} we have\\
when $p<\tilde{v},$
\begin{equation}\label{eq:542}\begin{aligned}
I_4
\leq& C||\bar{b}(x)||_{L^{\frac{\tilde{v}}{p-1}}(B_{2h})}||\eta||_{L^{p^{\ast}}(B_{2h})}^{p-1}||X\eta||_{L^p(B_{2h})}
\leq C||\bar{b}(x)||_{L^{\frac{\tilde{v}}{p-1}}(B_{2h})}||X\eta||_{L^p(B_{2h})}^p
\\ \leq & C h^{-p}|B_{2h}|||\bar{b}(x)||_{L^{\frac{\tilde{v}}{p-1}}(B_{2h})}
\leq C h^{-p}|B_{2h}|(||b(x)||_{L^{\frac{\tilde{v}}{p-1}}(B_{2h})}+1),\\
\end{aligned}\end{equation}
with $p^{\ast}= \frac{p\tilde{v}}{\tilde{v}-p};$ when $p=\tilde{v},$
\begin{equation}\label{eq:5422}\begin{aligned}
I_4
\leq&C||\bar{b}(x)||_{L^{\frac{\tilde{v}}{{\tilde{v}}-1-\varepsilon}}(B_{2h})}||\eta||_{L^{\frac{\tilde{v}(\tilde{v}-1)}{\varepsilon}}(B_{2h})}^{\tilde{v}-1}
||X\eta||_{L^{\tilde{v}}(B_{2h})}
\leq C||\bar{b}(x)||_{L^{\frac{\tilde{v}}{{\tilde{v}}-1-\varepsilon}}(B_{2h})}||X\eta||_{L^{\tilde{v}}(B_{2h})}^{\tilde{v}}
\\ \leq & C h^{-\tilde{v}}|B_{2h}|||\bar{b}(x)||_{L^{\frac{\tilde{v}}{{\tilde{v}}-1-\varepsilon}}(B_{2h})}
\leq C h^{-\tilde{v}}|B_{2h}|(||b(x)||_{L^{\frac{\tilde{v}}{{\tilde{v}}-1-\varepsilon}}(B_{2h})}+1).\\
\end{aligned}\end{equation}

Inserting these estimates into inequality \eqref{581},we obtain that
\begin{equation}\label{582}
||\eta Xv||_{L^p(B_{2h})}^p \leq ||\eta Xv||_{L^p(B_{4R})}^p  \leq  C (h^{-1}|B_{2h}|^{\frac{1}{p}}||\eta Xv||_{L^p(B_{2h})}^{p-1}+h^{-p}|B_{2h}|).
\end{equation}
By Lemma \ref{inez} and Lemma \ref{doubling}, we conclude that $$||Xv||_{L^p(B_h)} \leq||\eta Xv||_{L^p(B_{2h})}^{p}\leq C h^{-1}|B_{2h}|^{\frac{1}{p}} \leq C h^{-1}|B_{h}|^{\frac{1}{p}},$$
here C depends on $n,p,\varepsilon ,a,c_0,R,||d||_{\Omega},||c||_{\Omega},||b||_{\Omega},||u||_{L^{\infty}(\Omega)}.$\\

Based on Remark \ref{remark5.1} and H\"older inequality, we have
$$\begin{aligned}
\frac{1}{|B_h|}\int_{B_h}|v-v_B| dx
\leq &C h \left(\frac{1}{|B_h|}\int_{B_h} |Xv|^pdx\right)^{\frac{1}{p}}
\leq C. \end{aligned}$$
Due to $v= \log \bar{u} $ and $(B_{2R},\mu,d)$ is a homogeneous space, by Lemma \ref{jn}, there exist constants $\gamma_0^{'}>0$ and $C>0$ so that
$$(\frac{1}{|B_{2R}|}\int_{B_{2R}} {\bar u}^{\gamma_0^{'}}dx)^{\frac{1}{\gamma_0^{'}}} \leq C (\frac{1}{|B_{2R}|}\int_{B_{2R}} {\bar u}^{-\gamma_0^{'}}dx)^{-\frac{1}{\gamma_0^{'}}}.$$
Let
\begin{equation}\label{phi}
\Phi(\alpha,h):=(\int_{B_h} {\bar{u}}^\alpha dx)^{\frac{1}{\alpha}},
\end{equation} where $\alpha \neq 0$ and $l>0.$ So we get
\begin{equation}\label{eq:590}
\Phi(\gamma_0^{'},2R)\leq C{|B_{2R}|^{\frac{2}{\gamma_0^{'}}}}\Phi(-\gamma_0^{'},2R).
\end{equation}

Case\ 2: next, we consider the another case $\beta \neq 1-p$ and rewrite \eqref{572} as
\begin{equation}\label{591}
 \begin{aligned} \int_{B_{4R}} |\beta| |\eta Xv|^p dx
\leq &\int_{B_{4R}}  \bigg(|q|^p(1+|\beta|+c_0 \bar{u})\bar{d}(x)|\eta v|^p+|q|c(x)\eta v|\eta Xv|^{p-1}\\
 & +pa|q||v X\eta||\eta Xv|^{p-1}+ p|q|^p \bar{b}(x)|\eta v|^{p-1}|v X\eta|\bigg) dx.\\
\end{aligned}
\end{equation}
  When $p < \tilde{v},$  applying H\"older inequality and Lemma \ref{sobolev} we obtain that
\begin{equation}\label{eq:5222}\begin{aligned}
\int_{B_{4R}}\bar{d}(x)|\eta v|^pdx
\leq& C||\bar{d}(x)||_{L^{\frac{\tilde{v}}{p-\frac{\varepsilon}{2}}}(B_{4R})}||\eta v||_{L^p(B_{4R})}^{\frac{\varepsilon}{2}}||\eta v||_{L^{p^{\ast}}(B_{4R})}^{p-\frac{\varepsilon}{2}}\\
\leq& C|B_{2h}|^{\frac{\varepsilon}{2\tilde{v}}}||\bar{d}(x)||_{L^{\frac{\tilde{v}}{p-\varepsilon}}(B_{4R})}||\eta v||_{L^p(B_{4R})}^{\frac{\varepsilon}{2}}\left(||v X\eta||_{L^p(B_{4R})}^{p-\frac{\varepsilon}{2}}+||\eta Xv||_{L^p(B_{4R})}^{p-\frac{\varepsilon}{2}}\right)\\
\leq& C\left(||d(x)||_{L^{\frac{\tilde{v}}{p-{\varepsilon}}}(B_{4R})}+1\right)||\eta v||_{L^p(B_{4R})}^{\frac{\varepsilon}{2}}\left(||v X\eta ||_{L^p(B_{4R})}^{p-\frac{\varepsilon}{2}}+||\eta Xv||_{L^p(B_{4R})}^{p-\frac{\varepsilon}{2}}\right),\\
\end{aligned}\end{equation}
\begin{equation}\label{eq:5322}\begin{aligned}
&\int_{B_{4R}}c(x)|\eta v||\eta Xv|^{p-1}dx\leq ||c(x)||_{L^{\frac{\tilde{v}}{1-\frac{\varepsilon}{2}}}(B_{4R})}||\eta Xv||_{L^p(B_{4R})}^{p-1}||\eta v||_{L^{\frac{\tilde{v}p}{\tilde{v}-p+\frac{p\varepsilon}{2}}}(B_{4R})}  \\
\leq& ||c(x)||_{L^{\frac{\tilde{v}}{1-\frac{\varepsilon}{2}}}(B_{4R})}||\eta v||_{L^p(B_{4R})}^{\frac{\varepsilon}{2}} ||\eta v||_{p^{\ast}(B_{4R})}^{1-\frac{\varepsilon}{2}}  ||\eta Xv||_{L^p(B_{4R})}^{p-1}\\
\leq& C||c(x)||_{L^{\frac{\tilde{v}}{1-\varepsilon}}(B_{4R})} ||\eta v||_{L^p(B_{4R})}^{\frac{\varepsilon}{2}} \left(||v X\eta ||_{L^p(B_{4R})}^{1-\frac{\varepsilon}{2}}||\eta Xv||_{L^p(B_{4R})}^{p-1}+||\eta Xv||_{L^p(B_{4R})}^{p-\frac{\varepsilon}{2}}\right),\\
\end{aligned}\end{equation}
\begin{equation}\label{eq:5522}\begin{aligned}
\int_{B_{4R}} |\eta Xv|^{p-1} |vX\eta |dx  \leq&  ||v X\eta||_{L^p(B_{4R})}||\eta Xv||_{L^p(B_{4R})}^{p-1} ,\\
\end{aligned}\end{equation}
\begin{equation}\label{eq:54222}\begin{aligned}
&\int_{B_{4R}}\bar{b}(x) |\eta v|^{p-1} |v X\eta| dx\leq C||b(x)||_{L^{\frac{\tilde{v}}{p-1}}(B_{4R})}\left(||v X\eta||_{L^p(B_{4R})}^p +||vX\eta||_{L^p(B_{4R})} ||\eta Xv||_{L^p(B_{4R})}^{p-1}\right)
\\& \leq C\left(||b(x)||_{L^{\frac{\tilde{v}}{p-1}}(B_{4R})}+1\right)\left(||v X\eta||_{L^p(B_{4R})}^p +||vX\eta||_{L^p(B_{4R})} ||\eta Xv||_{L^p(B_{4R})}^{p-1}\right),\\
\end{aligned}\end{equation}
where $p^{\ast}=\frac{p\tilde{v}}{\tilde{v}-p}.$
Inserting these estimates into inequality \eqref{591}, we have
\begin{equation}\label{eq:122}
\begin{aligned}
|\beta|||\eta Xv||_{L^p(B_{4R})}^p \leq C\bigg(&|q|^p(1+|\beta|)\left(||\eta v||_{L^p(B_{4R})}^{\frac{\varepsilon}{2}}\left(||vX\eta ||_{L^p(B_{4R})}^{p-\frac{\varepsilon}{2}}+||\eta Xv||_{L^p(B_{4R})}^{p-\frac{\varepsilon}{2}}\right)\right)\\
&+|q|||\eta v||_{L^p(B_{4R})}^{\frac{\varepsilon}{2}}\left(|vX\eta ||_{L^p(B_{4R})}^{1-\frac{\varepsilon}{2}}||\eta Xv||_{L^p(B_{4R})}^{p-1}+||\eta Xv||_{L^p(B_{4R})}^{p-\frac{\varepsilon}{2}}\right)\\
&+|q||| vX\eta||_{L^p(B_{4R})} ||\eta Xv||_{L^p(B_{4R})}^{p-1}\\
&+|q|^p\left(|vX\eta ||_{L^p(B_{4R})}^p+||vX\eta ||_{L^p(B_{4R})} ||\eta Xv||_{L^p(B_{4R})}^{p-1}\right)\bigg),
\end{aligned}\end{equation}
here C depends on $n,p,\varepsilon ,a,c_0,R,||d||_{\Omega},||c||_{\Omega},||b||_{\Omega},||u||_{L^{\infty}(\Omega)}.$

We set
$$z=\frac{||\eta Xv||_{L^p(B_{4R})}}{||vX\eta ||_{L^p(B_{4R})}},\quad \xi=\frac{||\eta v||_{L^p(B_{4R})}}{||vX\eta ||_{L^p(B_{4R})}},$$
then we get
$$\begin{aligned}
|z|^p \leq & C\bigg(z^{p-\frac{\varepsilon}{2}}\xi^{\frac{\varepsilon}{2}}\left(1+{|\beta|}^{-1}\right)\left(|q|+|q|^p\right)+z^{p-1}\left(1+{|\beta|}^{-1}\right)\left(|q|\xi^{\frac{\varepsilon}{2}}+|q|+|q|^p\right)\\
&+\left(1+{|\beta|}^{-1}\right)\left(|q|^p\xi^{\frac{\varepsilon}{2}}+|q|^p\right)\bigg).
\end{aligned}$$
By  Lemma \ref{inez}, we can get the estimate $z\leq C{(1+{|\beta|}^{-1})}^{\frac{2}{\varepsilon} }(1+|q|^{\frac{2p}{\varepsilon} })(1+\xi),$ that is
$$||\eta Xv||_{L^p(B_{4R})}\leq C{\left(1+{|\beta|}^{-1}\right)}^{\frac{2}{\varepsilon}}\left(1+|q|^{\frac{2p}{\varepsilon}}\right)\left(||vX\eta ||_{L^p(B_{4R})}+||\eta v||_{L^p(B_{4R})}\right). $$

Setting $h_{\nu}=R(1+2^{-\nu}),$ for $\nu=0,1,2,...$, by Lemma \ref{cutoff}, we choose $\eta$ so that for any $R\leq h_{\nu+1}<h_{\nu}\leq 2R,$
$$\eta(x) =\begin{cases}
1  \  \ & if  \ \ x \in B_{h_{\nu+1}},\\
0   \ \ &if   \ \ x \notin B_{h_{\nu}},
\end{cases}$$
and we have $|X\eta| \leq \frac{2^{\nu+1}C}{R}.$ Applying Lemma \ref{sobolev} we obtain that
\begin{equation}\label{eq:55}\begin{aligned}
||v||_{L^{p^{\ast}}(B_{h_{\nu+1}})}\leq& ||\eta v||_{L^{p^{\ast}}(B_{4R})} \leq  ||\eta Xv||_{L^p(B_{4R})}+|| v X\eta||_{L^p(B_{4R})} \\ \leq& C{\left(1+{|\beta|}^{-1}\right)}^{\frac{2}{\varepsilon}}\left(1+|q|^{\frac{2p}{\varepsilon}}\right)(2^{\nu+1}R^{-1}) ||v||_{L^p(B_{h_{\nu}})}.
\end{aligned}\end{equation}
Then take the $q$ th root of two sides and set $\gamma=pq=p+\beta-1$ and $\chi=\frac{\tilde{v}}{{\tilde{v}}-p}.$ Combining $v={\bar{u}}^q$ and \eqref{phi} we obtain that
\begin{equation}\label{eq:56}
\Phi(\chi \gamma,h_{\nu+1})\leq\left( C{\left(1+{|\beta|}^{-1}\right)}^{\frac{2}{\varepsilon}}\left(1+\gamma^{\frac{2p}{\varepsilon}}\right)\left(2^{\nu+1}R^{-1}\right)\right)^{\frac{p}{\gamma}} \Phi(\gamma,h_{\nu}), \quad  q>0,
\end{equation}
\begin{equation}\label{eq:57}
\Phi(\chi \gamma,h_{\nu+1})\geq \left( C{\left(1+{|\beta|}^{-1}\right)}^{\frac{2}{\varepsilon}}\left(1+|\gamma|^{\frac{2p}{\varepsilon}}\right)\left(2^{\nu+1}R^{-1}\right)\right)^{\frac{p}{\gamma}} \Phi(\gamma,h_{\nu}), \quad  q<0.
\end{equation}

The following we will iterate the inequalities \eqref{eq:56} by setting that $\gamma_{\nu}=\chi^{\nu}\gamma_0$ for $\nu=0,1,2,...,$  and  we choose $0<\gamma_0 <\gamma_0^{'}$ such that $\gamma=p-1$ in which case $\beta =0 $ lies midway between two iterates $\gamma_{\nu}$ and $\gamma_{\nu+1},$ which guarantees an estimate $$|\beta|=|\gamma-(p-1)|\geq \frac{p(p-1)}{2\tilde {v}-p}.$$
Then for all $\nu,$ ${(1+{|\beta|}^{-1})}^{\frac{2}{\varepsilon}} \leq C,$ where $C$ only depends on $n,p,\varepsilon.$  Hence from \eqref{eq:56} we have
\begin{equation}\label{eq:59}
\begin{aligned}
\Phi(\gamma_{\nu+1},h_{\nu+1})
\leq &[ CR^{-1}{(1+\gamma_{\nu}^{\frac{2p}{\varepsilon}})2^{\nu+1}}]^{\frac{p}{\gamma_{\nu}}} \Phi(\gamma_{\nu},h_{\nu})
=[ CR^{-1}{(1+\chi^{\frac{2p\nu}{\varepsilon}}\gamma_0^{\frac{2p}{\varepsilon}})2^{\nu+1}]^{ p\chi^{-\nu}\gamma_0^{-1}}} \Phi(\gamma_{\nu},h_{\nu})\\
\leq &(CR^{-1})^{{\frac{p}{\gamma_0}}\sum \limits_{\nu=0}^{\infty}\chi^{-\nu}}2^{\frac{p}{\gamma_0}\sum\limits_{\nu=0}^{\infty}(\chi^{-\nu}+\nu\chi^{-\nu})}
(1+\chi^{\frac{2p^2}{\varepsilon \gamma_0}\sum\limits_{\nu=0}^{\infty}\nu \chi^{-\nu}}\gamma_0^{\frac{2p^2}{\varepsilon \gamma_0}\sum\limits_{\nu=0}^{\infty}\chi^{-\nu}})\Phi(\gamma_0,2R)\\
\leq &C(R^{-1})^{\frac{\tilde{v}}{\gamma_0}} \Phi(\gamma_0,2R).
\end{aligned}
\end{equation}

The above inequality holds based on $\chi=\frac{\tilde{v}}{{\tilde{v}}-p}>1.$  As $\nu\rightarrow \infty,$  by \eqref{eq:59} we conclude that
\begin{equation}\label{eq:591}
 \sup\limits_{B_R}\bar{u}=\Phi(\infty,R)\leq C(R^{-1})^{\frac{\tilde{v}}{\gamma_0}} \Phi(\gamma_0,2R) .
\end{equation}
Similarly,  by iterating \eqref{eq:57} as before, with the only modification that $\gamma_{\nu}=-\chi^{\nu}\gamma_0,$ we yield the result
\begin{equation}\label{eq:592}
 \inf\limits_{B_R}\bar{u}=\Phi(-\infty,R)\geq C^{-1}R^{\frac{\tilde{v}}{\gamma_0}} \Phi(-\gamma_0,2R) .
\end{equation}

Finally by \eqref{eq:590}, \eqref{eq:591}, \eqref{eq:592} and  H\"older inequality, we have
\begin{equation}\label{CR}
\begin{aligned}
\sup\limits_{B_R}\bar{u}\leq &C(R^{-1})^{\frac{\tilde{v}}{\gamma_0}} \Phi(\gamma_0,2R)
\leq C{|B_{2R}|^{\frac{\gamma_0^{'}-\gamma_0}{\gamma_0^{'}\gamma_0}}}(R^{-1})^{\frac{\tilde{v}}{\gamma_0}} \Phi(\gamma_0^{'},2R)\\
\leq &C{|B_{2R}|^{\frac{\gamma_0^{'}+\gamma_0}{\gamma_0^{'}\gamma_0}}} (R^{-1})^{\frac{\tilde{v}}{\gamma_0}}\Phi(-\gamma_0^{'},2R)
\leq C{|B_{2R}|^{\frac{2}{\gamma_0}}} (R^{-1})^{\frac{\tilde{v}}{\gamma_0}} \Phi(-\gamma_0,2R)\\
\leq &C{|B_{2R}|^{\frac{2}{\gamma_0}}} (R^{-1})^{\frac{2\tilde{v}}{\gamma_0}} \inf\limits_{B_R}\bar{u}
= C({\frac{|B_{2R}|}{R^{\tilde{v}}}})^{\frac{2}{\gamma_0}}\inf\limits_{B_R}\bar{u}=C\inf\limits_{B_R}\bar{u}.
\end{aligned}\end{equation}
Due to $\bar{u}=u+\tilde k(R),$ we conclude that $$\sup\limits_{B_R} u\leq C(\inf\limits_{B_R} u+\tilde k(R)).$$

When $p = \tilde{v},$ Inequalities \eqref{eq:5222}, \eqref{eq:5322} and \eqref{eq:5422}  are transformed into the following three expressions in sequence.
\begin{equation}\label{eq:52122}\begin{aligned}
\int_{B_{4R}}\bar{d}(x)|\eta v|^{\tilde{v}}dx
\leq& C||\bar{d}(x)||_{L^{\frac{\tilde{v}}{\tilde{v}-\frac{\varepsilon}{2}}}(B_{4R})}||\eta v||_{L^{\tilde{v}}(B_{4R})}^{\frac{\varepsilon}{4}}||\eta v||_{L^{\frac{4\tilde{v}^2-\tilde{v}\varepsilon}{\varepsilon}}(B_{4R})}^{{\tilde{v}}-\frac{\varepsilon}{4}}\\
\leq& C|B_{2h}|^{\frac{\varepsilon}{2\tilde{v}}}||\bar{d}(x)||_{L^{\frac{\tilde{v}}{{\tilde{v}}-\varepsilon}}(B_{4R})}||\eta v||_{L^{\tilde{v}}(B_{4R})}^{\frac{\varepsilon}{4}}(||v X\eta||_{L^{\tilde{v}}(B_{4R})}^{{\tilde{v}}-\frac{\varepsilon}{4}}+||\eta Xv||_{L^{\tilde{v}}(B_{4R})}^{{\tilde{v}}-\frac{\varepsilon}{4}})\\
\leq& C(||d(x)||_{L^{\frac{\tilde{v}}{{\tilde{v}}-{\varepsilon}}}(B_{4R})}+1)||\eta v||_{L^{\tilde{v}}(B_{4R})}^{\frac{\varepsilon}{4}}(||v X\eta ||_{L^{\tilde{v}}(B_{4R})}^{{\tilde{v}}-\frac{\varepsilon}{4}}+||\eta Xv||_{L^{\tilde{v}}(B_{4R})}^{{\tilde{v}}-\frac{\varepsilon}{4}}),\\
\end{aligned}\end{equation}
\begin{equation}\label{eq:53122}\begin{aligned}
&\int_{B_{4R}}c(x)|\eta v||\eta Xv|^{\tilde{v}-1}dx\leq ||c(x)||_{L^{\frac{\tilde{v}}{1-\frac{\varepsilon}{2}}}(B_{4R})}||\eta Xv||_{L^{\tilde{v}}(B_{4R})}^{{\tilde{v}}-1}||\eta v||_{L^{\frac{2\tilde{v}}{\varepsilon}}(B_{4R})}  \\
\leq& ||c(x)||_{L^{\frac{\tilde{v}}{1-\frac{\varepsilon}{2}}}(B_{4R})}||\eta v||_{L^{\tilde{v}}(B_{4R})}^{\frac{\varepsilon}{4}} ||\eta v||_{L^{\frac{\tilde{v}(4-\varepsilon)}{\varepsilon}}(B_{4R})}^{1-\frac{\varepsilon}{4}}  ||\eta Xv||_{L^{\tilde{v}}(B_{4R})}^{{\tilde{v}}-1}\\
\leq& C||c(x)||_{L^{\frac{\tilde{v}}{1-\varepsilon}}(B_{4R})} ||\eta v||_{L^{\tilde{v}}(B_{4R})}^{\frac{\varepsilon}{4}} (||v X\eta ||_{L^{\tilde{v}}(B_{4R})}^{1-\frac{\varepsilon}{4}}||\eta Xv||_{L^{\tilde{v}}(B_{4R})}^{\tilde{v}-1}+||\eta Xv||_{L^{\tilde{v}}(B_{4R})}^{{\tilde{v}}-\frac{\varepsilon}{4}}),\\
\end{aligned}\end{equation}
\begin{equation}\label{eq:54122}\begin{aligned}
&\int_{B_{4R}} \bar{b}(x) |\eta v|^{{\tilde{v}}-1} |v X\eta| dx
\leq C(||b(x)||_{L^{\frac{\tilde{v}}{{\tilde{v}}-1-\varepsilon}}(B_{4R})}+1)(||v X\eta||_{L^{\tilde{v}}(B_{4R})}^{\tilde{v}}+||v X\eta ||_{L^{\tilde{v}}(B_{4R})}||\eta Xv||_{L^{\tilde{v}}(B_{4R})}^{{\tilde{v}}-1}).\\
\end{aligned}\end{equation}
Then  \eqref{eq:122} holds with
$ \frac{\varepsilon}{4}$
substituting for $ \frac{\varepsilon}{2}$.
\end{proof}

\begin{proof}[\textbf{Proof of Theorem \ref{Thmholderctn1}}]
For any $x \in \Omega,$  there exists a $R_0$ such that $B(x,{4R_0}) \subset \subset \Omega,$ $R_0\leq \frac{\rho_{\Omega}}{4},$ and
Theorem  \ref{Thm5.1} holds. For any $r\leq R_0$  set
$$M(4r)=\sup\limits_{B_{4r}} u,\ m(4r)=\inf\limits_{B_{4r}} u, \ \omega(4r)= M(4r)-m(4r),$$
$$v_1=M(4r)-u ,\ v_2=u-m(4r).$$
It is clear that $v_1,v_2$ both are negative and bounded in $B_{4r}$ and satisfy the equation respectively as follow
$$-div_{X} A_i(x,v_i, Xv_i)=B_i(x,v_i,Xv_i).$$
Meanwhile the similar structure conditions still hold with
$$b_i(x)=C(p) b(x) , \ e_i(x)= C(p) b (x)||u ||_{L^{\infty}(\Omega)}^{p-1}+e(x),\ d_i(x)=C(p)d(x), $$
$$g_i(x)=C(p) d(x) ||u ||_{L^{\infty}(\Omega)}^{p}+g(x) ,\ f_i(x)= C(p) d(x) ||u ||_{L^{\infty}(\Omega)}^{p-1}+f(x).$$

 Apply the Theorem \ref{Thm5.1} to $v_1$ and $v_2$ in the open ball $B_{4r}$ with \begin{equation}\label{eq:29}\tilde k(r)=(|B_{r}|^{\frac{\varepsilon}{2\tilde v}}||e||_{B_{4r}}+|B_{r}|^{\frac{\varepsilon}{2\tilde v}}||f||_{B_{4r}})^{\frac{1}{p-1}}+(|B_{r}|^{\frac{\varepsilon}{2\tilde v}}||g||_{B_{4r}})^{\frac{1}{p}},\end{equation}
and then \eqref{bd}, \eqref{eq:542}, \eqref{eq:5422}, \eqref{eq:54222} and \eqref{eq:54122} in the proof of  Theorem \ref{Thm5.1} turn into
$$
||\bar b_i(x)||_{{B_{4r}}}\leq ||b_i(x)||_{{B_{4r}}} +|B_{r}|^{-\frac{\varepsilon}{2\tilde v}},\;\;\;||\bar d_i(x)||_{B_{4r}}\leq ||d_i(x)||_{B_{4r}} +2|B_{r}|^{-\frac{\varepsilon}{2\tilde v}},
$$
$$\begin{aligned}
\int_{B_{4r}} \bar{b}_i(x)\eta^{p-1}|X\eta| dx\leq & C||\bar{b}_i(x)||_{L^\frac{\tilde{v}}{p-1-\frac{\varepsilon}{2}}(B_{2h})}|| \eta||^{p-1}_{L^\frac{p(p-1)\tilde{v}}{p(\tilde{v}-p+1+\frac{\varepsilon}{2})}(B_{2h})}||X\eta||_{L^p(B_{2h})} \\ \leq & C||\bar{b}_i(x)||_{L^\frac{\tilde{v}}{p-1-\frac{\varepsilon}{2}}(B_{2h})}||X\eta||_{L^p(B_{2h})}^p
\leq C|B_{2h}|^{\frac{\varepsilon}{2\tilde v}}||\bar{b}_i(x)||_{L^\frac{\tilde{v}}{p-1-{\varepsilon}}(B_{2h})}||X\eta||_{L^p(B_{2h})}^p\\
\leq&C h^{-p}|B_{2h}|(||b_i(x)||_{L^\frac{\tilde{v}}{p-1-\varepsilon}(B_{2h})}+1),\\
\end{aligned}$$

$$\begin{aligned}
\int_{B_{4r}} \bar{b}_i(x)\eta^{\tilde{v}-1}|X\eta| dx
\leq&C||\bar{b}_i(x)||_{L^{\frac{\tilde{v}}{{\tilde{v}}-1-\frac{\varepsilon}{2}}}(B_{2h})}||\eta||_{L^{\frac{2\tilde{v}(\tilde{v}-1)}{\varepsilon}}(B_{2h})}^{\tilde{v}-1}
||X\eta||_{L^{\tilde{v}}(B_{2h})}
\\ \leq& C|B_{2h}|^{\frac{\varepsilon}{2\tilde v}}||\bar{b}_i(x)||_{L^{\frac{\tilde{v}}{{\tilde{v}}-1-\varepsilon}}(B_{2h})}||X\eta||_{L^{\tilde{v}}(B_{2h})}^{\tilde{v}}
\\ \leq & C h^{-\tilde{v}}|B_{2h}|^{1+\frac{\epsilon}{2\tilde{v}}}||\bar{b}_i(x)||_{L^{\frac{\tilde{v}}{{\tilde{v}}-1-\varepsilon}}(B_{2h})}
\leq C h^{-\tilde{v}}|B_{2h}|(||b_i(x)||_{L^{\frac{\tilde{v}}{{\tilde{v}}-1-\varepsilon}}(B_{2h})}+1).\\
\end{aligned}$$

$$\begin{aligned}
\int_{B_{4r}} \bar{b}_i(x) |\eta v|^{p-1} |vX\eta | dx \leq &||\bar{b}_i(x)||_{L^\frac{\tilde{v}}{p-1-\frac{\varepsilon}{2}}(B_{4r})}||\eta v||^{p-1}_{L^\frac{p(p-1)\tilde{v}}{p(\tilde{v}-p+1+\frac{\varepsilon}{2})}(B_{4r})}||vX\eta ||_{L^p(B_{4r})}^p
\\ \leq & ||\bar{b}_i(x)||_{L^\frac{\tilde{v}}{p-1-\frac{\varepsilon}{2}}(B_{4r})}(||vX\eta ||_{L^p(B_{4r})}^p +||vX\eta ||_{L^p(B_{4r})} ||\eta Xv||_{L^p(B_{4r})}^{p-1})\\
\leq& C(||b_i(x)||_{L^\frac{\tilde{v}}{p-1-\varepsilon}(B_{4r})}+1)(||vX\eta ||_{L^p(B_{4r})}^p +||vX\eta ||_{L^p(B_{4r})} ||\eta Xv||_{L^p(B_{4r})}^{p-1}),\\
\end{aligned}$$
$$\begin{aligned}
\int_{B_{4r}} \bar{b}_i(x)|\eta v|^{{\tilde{v}}-1} |v X\eta| dx \leq & C||\bar{b}_i(x)||_{L^{\frac{\tilde{v}}{{\tilde{v}}-1-\frac{\varepsilon}{2}}}(B_{4r})}||\eta v||_{L^{\frac{2\tilde{v}(\tilde{v}-1)}{\varepsilon}}(B_{4r})}^{\tilde{v}-1}
||vX\eta||_{L^{\tilde{v}}(B_{4r})}\\
\leq & C(||b_i(x)||_{L^{\frac{\tilde{v}}{{\tilde{v}}-1-\varepsilon}}(B_{4r})}+1)(||v X\eta||_{L^{\tilde{v}}(B_{4r})}^{\tilde{v}}+||v X\eta ||_{L^{\tilde{v}}(B_{4r})} ||\eta Xv||_{L^{\tilde{v}}(B_{4r})}^{{\tilde{v}}-1}).\\
\end{aligned}$$
which will not affect the derivation of the conclusion. And since $\Omega$ is equiregular, \eqref{CR} can be corrected to
$$\sup\limits_{B_{r}}\bar{u}\leq  C\left({\frac{|B_{2r}|}{r^{\tilde{v}}}}\right)^{\frac{2}{\gamma_0}}\inf\limits_{B_{r}}\bar{u}
=C\inf\limits_{B_{r}}\bar{u},$$
where the constant $C$ is independent of $r.$
So we obtain that
\begin{equation}\label{eq:593}
M(4r)-m(r)\leq C\left(M(4r)-M(r)+\tilde k(r)\right),
\end{equation}
\begin{equation}\label{eq:594}
M(r)-m(4r)\leq C\left(m(r)-m(4r)+\tilde k(r)\right).
\end{equation}
For convenience, we denote $$k_0:= \left(||e_i||_{\Omega}+||f_i||_{\Omega}\right)^{\frac{1}{p-1}}+||g_i||_{\Omega}^{\frac{1}{p}}.$$
By \eqref{eq:29} we note that
$$\tilde k(r)\leq |B_{r}|^{\tilde \rho}k_0 \leq C k_0 r^{\tilde{v}\tilde \rho},$$
here $\tilde \rho={\frac{\varepsilon}{2p}}$ when $ |B_{r}|\leq 1;$ $\tilde \rho={\frac{\varepsilon}{2(p-1)}}$ when $ |B_{r}|\geq 1.$

Adding \eqref{eq:593} and\eqref{eq:594}  we yield that
\begin{equation}\label{eq:595}
 \omega(r)\leq {\frac{C-1}{C+1}}\omega{(4r)}+{\frac{2Ck_0}{C+1}r^{Q \tilde \rho}}.
\end{equation}
Applying Lemma \ref{lemma3.7}, we get that there exist $C>0,$ $\alpha \in (0,1)$ such that
$$\omega(4r)\leq C(|| u||_{L^{\infty}(\Omega)}+k_0)({4r})^{\alpha},$$
that is $$[u]_{\alpha,B_{4r}}\leq C(|| u||_{L^{\infty}(\Omega)}+k_0).$$

Finally for any compact subset $\Omega_0 \subset \Omega$, by finite covering theorem we may complete the conclusion.
\end{proof}

\section{Application}
In this section, we give applications to   higher step Grushin type operator.
For the Grushin type opterator in $\mathbb{R}^n$ as \eqref{eq:34},
 we consider the equation \begin{equation}\label{eq:27}P_{k'}(u)=f(x,z). \end{equation}
Since \eqref{eq:35}, and
then  \eqref{eq:27} implies
  \begin{equation}\label{eq:28}  \sum_{i=1}^{n}X^{*}_{i}(X_iu)=f(x,z).\end{equation}


\begin{proof}[\textbf{proof of Corollary \ref{thm4.1}}]
It's obviously that \eqref{eq:28} satisfies the condition in Theorem \ref{Thmremovsing} with $$p=2, \ \theta=1, \ a=1,$$ $$b(x,z)=0,\ c(x,z)=0,\   e(x,z)=0, \ d(x,z)=0, \ g(x,z)=0.$$
Then through Theorem \ref{Thmremovsing}, we conclude.
\end{proof}

\begin{proof}[\textbf{proof of Theorem \ref{thm4.2}}]
Through calculation, we get $cap_{h+l(k'+1)}(\{(0,z')\})=0$. Therefore, we only need to verify $u\in L^{\frac{h+l(k'+1)}{h+l(k'+1)-2}(1+\delta)}_{loc}( \mathbb{R}^n)$ and apply Theorem \ref{thm4.1} to conclude.
Since $h\geq 2$, $l\geq 1$ and $k'\geq 0$,
 when $|(\zeta, \varsigma)|$ is enough small we have
$$\begin{aligned}
|(x,z-z')|^{2(k'+1)}= &(|x|^2+|z-z'|^2)^{(k'+1)}
\leq  C_2 ( |x|^{2(k'+1)}+ |z-z'|^{2(k'+1)}) \\  \leq & C_2(|x|^{2(k'+1)}+(k'+1)|z-z'|^2).\end{aligned}$$
 Then when $|(\zeta, \varsigma)|$ is small enough, we have  if $\tilde{\delta}\geq \frac{l}{2}+\frac{h-2}{2(k'+1)}$, $
u(x,z)\in L^{\frac{h+l(k'+1)}{h+l(k'+1)-2}(1+\delta)}_{loc}( \mathbb{R}^n)$ for some $\delta>0$;
 if $\tilde{\delta}<\frac{l}{2}+\frac{h-2}{2(k'+1)}$,
 $$\begin{aligned}
| u(x,z)|\leq&  \frac{C}{ |(x, z-z')|^{l(k'+1)+h-2-2(k'+1)\tilde{\delta}}}.
 \end{aligned}$$

Since $\tilde{\delta}> \frac{lk'}{2(k'+1)}\frac{h+l(k'+1)-2}{h+l(k'+1)}$,
 there is a $\delta>0$ such that $$({l(k'+1)+h-2-2(k'+1)\tilde{\delta}})\times{ \frac{(h+l(k'+1))(1+\delta)}{h+l(k'+1)-2} }< n,$$ which yields
$u(x,z)\in L^{\frac{h+l(k'+1)}{h+l(k'+1)-2}(1+\delta)}_{loc}( \mathbb{R}^n)\ \text{for}\ \text{some}  \; \delta >0.$
Finally, by Theorem \ref{thm4.1}, we conclude.
\end{proof}
\begin{proof}[\textbf{proof of Corollary \ref{Thm1.7}}]
It's obviously that \eqref{eq:28} satisfies the condition in Theorem \ref{Thm5.1}, \ref{Thmholderctn1} with $$p=2<3\leq n\leq \tilde{v}, \ c_0=1, \ a=1,$$ $$b(x,z)=0,\ c(x,z)=0,\   e(x,z)=0, \ d(x,z)=0, \ g(x,z)=0.$$
Then through Theorem \ref{Thm5.1}, \ref{Thmholderctn1}, we conclude.
\end{proof}

\subsection*{Acknowledgments}
The authors are grateful to the referees for their careful reading and valuable comments.


\begin{thebibliography}{00}
\bibitem{BFI} W. Bauer,   K. Furutani and   C. Iwasaki;
 Fundamental solution of a higher step {G}rushin type operator,
 \emph{Advances in Mathematics}, \textbf{271} (2015), 188-234.
 \bibitem{AB} A. Bj\"orn;
Removable singularities for bounded $\mathcal{A}$-(super)harmonic
              and quasi(super)harmonic functions on weighted {$\bold{R}^n$},
\emph{Nonlinear Analysis. Theory, Methods $\&$ Applications }, \textbf{222} (2022), Paper No. 112907, 16 pp.


\bibitem{SAM}S. Biagi, A. Bonfiglioli and M. Bramanti;
Global estimates for the fundamental solution of homogeneous
              {H}\"ormander operators,
\emph{Annali di Matematica Pura ed Applicata. Series IV}, \textbf{201} (2022), 1875-1934.

\bibitem{LC}L. Capogna, D. Danielli and N. Garofalo;
The geometric {S}obolev embedding for vector fields and the
              isoperimetric inequality,
\emph{Communications in Analysis and Geometry}, \textbf{2} (1994), 203-215.




\bibitem{LC1}L. Capogna, D. Danielli and N. Garofalo;
Capacitary estimates and the local behavior of solutions of
              nonlinear subelliptic equations,
\emph{American Journal of Mathematics}, \textbf{118} (1996), 1153-1196.




\bibitem{LDN} L. Capogna, D. Danielli and  N. Garofalo;
An embedding theorem and the {H}arnack inequality for
              nonlinear subelliptic equations,
\emph{Communications in Partial Differential Equations}, \textbf{18} (1993), 1765-1794.

\bibitem{JN} Duong Quoc. Huy and  Luong Dang. Ky;
John-Nirenberg type inequalities for Musielak-Orlicz Campanato spaces on spaces of homogeneous type,
\emph{Vietnam Journal of Mathematics}, \textbf{47} (2019), 461-476.


\bibitem{GSP} Hua Chen, Hong-Ge Chen, and Jin-Ning Li;
 Sharp embedding results and geometric inequalities for H\"{o}rmander vector fields,
 \emph{arXiv preprint arXiv:2404.19393}, (2024), 4-19.
 \bibitem{TDD} Tan Duc. Do ;
Harnack inequality for degenerate {H}ormander sub-elliptic
              operators with potentials,
\emph{Expositiones Mathematicae}, \textbf{40} (2022), 605-627.

 \bibitem{FF} F. Ferrari;
Harnack inequality for two-weight subelliptic {$p$}-{L}aplace
              operators,
\emph{Mathematische Nachrichten}, \textbf{279} (2006), 815-830.

\bibitem{GBFEMS} G. B. Folland and E. M. Stein;
Estimates for the {$\bar \partial \sb{b}$} complex and
              analysis on the {H}eisenberg group,
\emph{Communications on Pure and Applied Mathematics}, \textbf{27} (1974), 429-522.


\bibitem{GBF} G. B. Folland;
Subelliptic estimates and function spaces on nilpotent {L}ie
              groups,
\emph{Arkiv f\"or Matematik}, \textbf{13} (1975), 161-207.

 \bibitem{J}  B. Franchi, G. Lu and R. L. Wheeden;
Representation formulas and weighted Poincar\'e
              inequalities for H\"ormander vector fields,
\emph{Universit\'e  de Grenoble. Annales de l'Institut Fourier}, \textbf{45} (1995), 577-604.


 \bibitem{NGDMN}N. Garofalo and   D.M. Nhieu;
Isoperimetric and {S}obolev inequalities for
              {C}arnot-{C}arath\'eodory spaces and the existence of minimal
              surfaces, \emph{Communications on Pure and Applied Mathematics}, \textbf{49} (1996), 1081-1144.


  \bibitem{DN}  D. Gilbarg and N.S. Trudinger;
Elliptic partial differential equations of second order, \emph{Springer-Verlag, Berlin}, \textbf{224} (1983), xiii+513 pp.









\bibitem{LH}L. H\"ormander;
Hypoelliptic second order differential equations,
\emph{Acta Mathematica}, \textbf{119} (1967), 147-171.


\bibitem{DJ}D. Jerison;
The {P}oincar\'e{} inequality for vector fields satisfying
              {H}\"ormander's condition,
\emph{Duke Mathematical Journal}, \textbf{53} (1986), 503-523.

\bibitem{SJZL} Shanming Ji, Zongguang Li and Changjiang Zhu;
Removable singularities and unbounded asymptotic profiles of
              multi-dimensional {B}urgers equations,
\emph{Mathematische Annalen }, \textbf{391} (2025), 113-162.

 \bibitem{VMO}
A. E.  Kogoj and  E. Lanconelli;
 Liouville theorem for {$X$}-elliptic operators,
 \emph{Nonlinear Analysis. Theory, Methods $\&$ Applications.}, \textbf{70} (2009),2974-2985.


\bibitem{HKEU} H. Kozono and E. Ushikoshi and F. Wakabayashi;
Removability of time-dependent singularities of solutions to
              the {N}avier-{S}tokes equations,
\emph{Journal of Differential Equations}, \textbf{388} (2024), 59-81, 16 pp.

 \bibitem{GL}G. Lu;
Weighted {P}oincar\'e{} and {S}obolev inequalities for vector
              fields satisfying {H}\"ormander's condition and applications,
\emph{Revista Matem\'atica Iberoamericana}, \textbf{8} (1992), 367-439.



 \bibitem{GL} G. Lu;
On {H}arnack's inequality for a class of strongly degenerate
              {S}chr\"odinger operators formed by vector fields,
\emph{Differential and Integral Equations. An International Journal
              for Theory and Applications}, \textbf{7} (1994), 73-100.


\bibitem{GL1} G. Lu;
Embedding theorems into {L}ipschitz and {BMO} spaces and
              applications to quasilinear subelliptic differential
              equations,
\emph{Publicacions Matem\`atiques}, \textbf{40} (1996), 301-329.


\bibitem{MM1} M. Meier;
Removable singularities for weak solutions of quasilinear
              elliptic systems,
 \emph{Journal f\"ur die Reine und Angewandte Mathematik}, \textbf{344} (1983), 87-101.

\bibitem{MM} M. Meier;
 Boundedness and integrability properties of weak solutions of quasilinear elliptic systems,
 \emph{Journal f\"ur die Reine und Angewandte Mathematik}, \textbf{333} (1982), 191-200.

 \bibitem{AM} A. Montanari;
Harnack inequality for a subelliptic {PDE} in nondivergence
              form,
\emph{Nonlinear Analysis. Theory, Methods $\&$ Applications}, \textbf{109} (2014), 285-300.

\bibitem{JM}J. Moser;
A new proof of {D}e {G}iorgi's theorem concerning the
              regularity problem for elliptic differential equations,
\emph{Communications on Pure and Applied Mathematics }, \textbf{13} (1960), 457-468.


\bibitem{JM1}J. Moser;
On {H}arnack's theorem for elliptic differential equations,
\emph{Communications on Pure and Applied Mathematics }, \textbf{14} (1961), 577-591.
\bibitem{FN}F. Nicolosi,  I. V. Skrypnik  and I. I. Skrypnik;
Precise point-wise growth conditions for removable isolated
              singularities,
\emph{Communications in Partial Differential Equations}, \textbf{28} (2003), 677-696.


\bibitem{LPR} L. P. Rothschild  and E. M. Stein;
Hypoelliptic differential operators and nilpotent groups,
\emph{Acta Mathematica}, \textbf{137} (1976), 247-320.

\bibitem{ASC}A. S\'anchez-Calle;
Fundamental solutions and geometry of the sum of squares of
              vector fields,
\emph{Inventiones Mathematicae}, \textbf{78} (1984), 143-160.



 \bibitem{SER} J. Serrin;
 Local behavior of solutions of quasi-linear equations,
 \emph{Acta Mathematica}, \textbf{111} (1964), 252-279.





















\bibitem{LWPN}  Leyun Wu and Pengcheng  Niu;
Harnack inequalities for weighted subelliptic {$p$}-{L}aplace
              equations constructed by {H}\"ormander vector fields,
\emph{Mathematical Reports (Bucure\c sti)}, \textbf{19(69)} (2017), 313-337.

































\bibitem{BW}Bo  Wang;
Removable singularities for degenerate elliptic {P}ucci
              operator on the {H}eisenberg group,
\emph{Nonlinear Analysis. Theory, Methods $\&$ Applications }, \textbf{160} (2017), 177-190.









\end{thebibliography}
\end{document}